\documentclass[]{amsart}
\usepackage[all]{xy}
\usepackage{graphicx}
\usepackage{amscd,amsthm,amssymb,amsfonts,amsmath,euscript}
\theoremstyle{plain}
\newtheorem{thm}{Theorem}[section]
\newtheorem{lm}[thm]{Lemma}
\newtheorem{prop}[thm]{Proposition}
\newtheorem{cor}[thm]{Corollary}

\theoremstyle{definition}
\newtheorem{defn}[thm]{Definition}
\newtheorem{eg}[thm]{Example}

\theoremstyle{remark}
\newtheorem{remark}[thm]{Remark}
\newtheorem*{thank}{Acknowledgments}
\newcommand{\nc}{\newcommand}

\def\makeop#1{\expandafter\def\csname#1\endcsname
  {\mathop{\rm #1}\nolimits}\ignorespaces}
\makeop{Hom}   \makeop{End}   \makeop{Aut}   \makeop{Isom}  \makeop{Pic} 
\makeop{Gal}   \makeop{ord}   \makeop{Char}  \makeop{Div}   \makeop{Lie} 
\makeop{PGL}   \makeop{Corr}  \makeop{PSL}   \makeop{sgn}   \makeop{Spf}
\makeop{Spec}  \makeop{Tr}    \makeop{Nr}    \makeop{Fr}    \makeop{disc}
\makeop{Proj}  \makeop{supp}  \makeop{ker}   \makeop{im}    \makeop{dom}
\makeop{coker} \makeop{Stab}  \makeop{SO}    \makeop{SL}    \makeop{SL}
\makeop{Cl}    \makeop{cond}  \makeop{Br}    \makeop{inv}   \makeop{rank}
\makeop{id}    \makeop{Fil}   \makeop{Frac}  \makeop{GL}    \makeop{SU}
\makeop{Trd}   \makeop{Sp}    \makeop{Tr}    \makeop{Trd}   \makeop{diag}
\makeop{Res}   \makeop{ind}   \makeop{depth} \makeop{Tr}    \makeop{st}
\makeop{Ad}    \makeop{Int}   \makeop{tr}    \makeop{Sym}   \makeop{can}
\makeop{length}\makeop{SO}    \makeop{torsion} \makeop{GSp} \makeop{Ker}
\def\makebb#1{\expandafter\def
  \csname bb#1\endcsname{{\mathbb{#1}}}\ignorespaces}
\def\makebf#1{\expandafter\def\csname bf#1\endcsname{{\bf
      #1}}\ignorespaces} 
\def\makegr#1{\expandafter\def
  \csname gr#1\endcsname{{\mathfrak{#1}}}\ignorespaces}
\def\makescr#1{\expandafter\def
  \csname scr#1\endcsname{{\EuScript{#1}}}\ignorespaces}
\def\makecal#1{\expandafter\def\csname cal#1\endcsname{{\mathcal
      #1}}\ignorespaces} 
\def\doLetters#1{#1A #1B #1C #1D #1E #1F #1G #1H #1I #1J #1K #1L #1M
                 #1N #1O #1P #1Q #1R #1S #1T #1U #1V #1W #1X #1Y #1Z}
\def\doletters#1{#1a #1b #1c #1d #1e #1f #1g #1h #1i #1j #1k #1l #1m
                 #1n #1o #1p #1q #1r #1s #1t #1u #1v #1w #1x #1y #1z}
\doLetters\makebb   \doLetters\makecal  \doLetters\makebf
\doLetters\makescr 
\doletters\makebf   \doLetters\makegr   \doletters\makegr
     \def\qed{\qedmark\medbreak}%
\def\qedmark{{\enspace\vrule height 6pt width 5pt depth 1.5pt}}%
    \def\setminus{\smallsetminus}

\normalsize

\makeop{Bl}
\newcommand{\Z}{\mathbb Z}
\newcommand{\Q}{\mathbb Q}

\newcommand{\F}{\mathbb F}

\nc{\embed}{\hookrightarrow}
\newcommand{\ES}{\bf{ES}}

\newcommand{\dieu}{Dieudonn\'{e} }

\nc{\ol}{\overline}
\nc{\wt}{\widetilde}
\nc{\opp}{\mathrm{opp}}

\makeop{Ram}
\makeop{Rep}

\newcommand\gfrac[2]{\genfrac{}{}{0pt}{}{#1}{#2}}

\newcommand\isomarrow{\stackrel{\cong}{\longrightarrow}}

\begin{document}
\renewcommand{\thefootnote}{\fnsymbol{footnote}}
\setcounter{footnote}{-1}
\numberwithin{equation}{section}
%\numberwithin{section}{chapter}

\title[The supersingular locus in the Iwahori case]{The supersingular locus
in Siegel modular varieties with Iwahori level structure}
\author{Ulrich G\"ortz}
\address[G\"ortz]{
Universit\"at Duisburg-Essen\\
Institut f\"{u}r Experimentelle Mathematik\\
Ellernstr.~29\\
45326 Essen\\
Germany}
\email{ulrich.goertz@uni-due.de}
\thanks{G\"{o}rtz was partially supported by a Heisenberg grant and by the
SFB/TR 45 ``Periods, Moduli Spaces and Arithmetic of Algebraic Varieties''
of the DFG (German Research Foundation)}
\author{Chia-Fu Yu}
\address[Yu]{
Institute of Mathematics \\
Academia Sinica \\
128 Academia Rd.~Sec.~2, Nankang\\ 
Taipei, Taiwan \\ and NCTS (Taipei Office)}
\email{chiafu@math.sinica.edu.tw}
\thanks{Yu was partially supported by grants 
NSC 97-2115-M-001-015-MY3 and AS-98-CDA-M01}.
% NSC 96-2115-M-001-001}.
% \date{January 24, 2007.  The research is partially supported by NSC
%  96-2115-M-001-001.}

\begin{abstract}
We study moduli spaces of abelian varieties in positive characteristic,
more specifically the moduli space of principally polarized abelian
varieties on the one hand, and the analogous space with Iwahori type level
structure, on the other hand. We investigate the Ekedahl-Oort
stratification on the former, the Kottwitz-Rapoport stratification on the
latter, and their relationship.
In this way, we obtain structural results about the supersingular locus in
the case of Iwahori level structure, for instance a formula for its
dimension in case $g$ is even.
\end{abstract}
 
\maketitle

%\tableofcontents   % Table of Contents

%\def\char{{\rm char}},                   

% --------------------------------------------------------------------------------

\section{Introduction}

Fix an integer $g$, a prime number $p$, and let $k$ be an algebraic closure
of the field $\mathbb F_p$ with $p$ elements. Denote by $\mathcal A_g$ the
moduli space of principally polarized abelian varieties of dimension $g$
(with a suitable level structure away from $p$), and by $\mathcal A_I$ the
moduli space of abelian varieties of dimension $g$ ``with Iwahori level
structure at $p$'', i.~e.~the space of chains $A_0 \rightarrow A_1
\rightarrow \cdots \rightarrow A_g$ of isogenies of order $p$, satisfying
certain further conditions. See Subsection \ref{notation} for details.

Whereas the space $\mathcal A_g$ is a well-studied object, much less is
known about $\mathcal A_I$. For instance, there are a number of structural
results about the supersingular locus $\mathcal S_g$ in $\mathcal A_g$, but
not even the dimension of the supersingular locus $\mathcal S_I$ inside
$\mathcal A_I$ is known in general. We prove (see
Corollary~\ref{dim_ss_locus}, 
Propositions~\ref{irred_comp_ss} and \ref{irred_comp_ss_g_even},
and~\cite{goertz-yu}, Proposition~4.6 and Theorem~6.3):

\begin{thm} \label{first_thm}
\begin{enumerate}
\item
If $g$ is even, then $\dim \mathcal S_I = \frac{g^2}{2}$.\\
If $g$ is odd, then
\[
\frac{g(g-1)}{2} \le \dim \mathcal S_I \le \frac{(g+1)(g-1)}{2}.
\]
\item
Suppose $g$ is even. Then every top-dimensional irreducible component of
$\mathcal S_I$ is isomorphic to the full flag variety of the group
$\Sp_{2\frac{g}{2}} \times \Sp_{2\frac{g}{2}}$ (over $k$).
\item
Suppose $g$ is odd. Then $\mathcal S_I$ has irreducible components
which are isomorphic to the flag variety of $\SL_g$ (over $k$). (So these
components are of dimension $g(g-1)/2$.)
\end{enumerate}
\end{thm}

However, one must keep in mind that in the Iwahori case, as soon as $g \ge
2$, the supersingular locus is not equidimensional.

Let us mention that the supersingular locus is expected to be of
considerable interest from the point of view of the Langlands program. In
fact, the supersingular locus is the unique closed Newton stratum, and the
Newton stratification, which is given by the isogeny type of the underlying
$p$-divisible group, provides a natural way to split up the space $\mathcal
A_I$, or more generally the special fiber of a Shimura variety of PEL type.
This can be seen, for instance, by looking at Kottwitz' method of counting
the points of the special fiber of a ``simple Shimura variety'', and by his
conjecture on the cohomology of Rapoport-Zink spaces \cite{rapoport:icm}
which should provide a local approach to the local Langlands
correspondence.  Conjecturally, in the non-supersingular locus, induced
representations should be realized, whereas in the supersingular locus the
supercuspidal representations are found. (Of course, here one will have to
use deeper level structure.) See \cite{boyer}, \cite{fargues-mantovan},
\cite{haines:clay}, \cite{harris-taylor:llc}, \cite{rapoport:guide} for
details and further references.

The key method to obtain our results is to study the Kottwitz-Rapoport
stratification (KR stratification) on $\mathcal A_I$. Insofar, the current
paper is a continuation of \cite{goertz-yu}, and in fact, we prove the two
conjectures we made in (the preprint version of) \cite{goertz-yu}. 
The first one concerns the dimension of the $p$-rank $0$ locus in
$\mathcal A_I$, and we have (Theorem \ref{dim_prk0}):

\begin{thm}
The dimension of the $p$-rank $0$ locus in $\mathcal A_I$ is $[g^2/2]$.
\end{thm}

An important ingredient of the proof is the result of Ng\^o and
Genestier that the $p$-rank is constant on KR strata (together with a
formula for the $p$-rank of a given stratum), but the task to make a list
of maximal strata of $p$-rank $0$ and to relate the formula of Ng\^o and
Genestier to the dimension of these strata, though purely combinatorial,
turns out to be quite involved. The theorem in particular gives us an
upper bound on the dimension of $\mathcal S_I$. Again, note that the
$p$-rank $0$ locus is usually not equidimensional.

Note that for $g\ge 3$, the supersingular locus is not a union of KR
strata. For instance, the stratum given by $x=s_3s_1s_0$ in $\mathcal
A_I$, $g=3$, intersects the supersingular locus without being
contained in it (\cite{yu:KR}). 
We call a KR stratum \emph{supersingular} if it is contained
in the supersingular locus. 

The second conjecture says that all supersingular KR strata a very specific
form, and hence can be described in very concrete terms. Let us recall
the following definitions: 

\begin{defn}
\begin{enumerate}
\item
An abelian variety over a field is called \emph{superspecial}, if it is isomorphic to a product of supersingular elliptic curves over some extension field.
\item
A KR stratum $\calA_x$ is called \emph{superspecial}, if there exists $i$
with $0\le i \le [g/2]$ such that for all $k$-valued points
$(A_\bullet, \lambda_0,\lambda_g)$ in $\calA_x$, the abelian varieties
$A_i$ and $A_{g-i}$ are superspecial, and the isogeny $A_i\rightarrow
A_{g-i}^\vee$ is isomorphic to the Frobenius morphism of
$A_i$. (Cf.~\cite{goertz-yu}, Section 4.) 
\end{enumerate}
\end{defn}

In particular, every superspecial abelian variety is supersingular, and every
superspecial KR stratum is supersingular. The main result of \cite{goertz-yu}
is an explicit geometric description of these superspecial KR strata in terms
of Deligne-Lusztig varieties.
Here we prove (Corollary \ref{ssi_implies_ssp}):

\begin{thm}
All supersingular KR strata are superspecial.
\end{thm}

Looking at superspecial strata, we get a lower bound on $\dim \mathcal S_I$, and in fact a more
precise version of Theorem \ref{first_thm} is that the irreducible
components referred to in (2) and (3) are closures of components of
superspecial KR strata of the maximal dimension.
Along the way, we prove the following results about KR strata, which are of
independent interest (see Theorems \ref{KR_quasiaffine} and
\ref{ssp_or_irred}).

\begin{thm}
\begin{enumerate}
\item
All KR strata are quasi-affine.
\item
All KR strata which are not superspecial, are connected.
\end{enumerate}
\end{thm}

The key ingredient which we did not use in \cite{goertz-yu} is a comparison
of the KR stratification with the Ekedahl-Oort stratification on $\mathcal
A_g$. This allows us to use results about the latter: for instance, part
(2) of the previous theorem ultimately follows from the fact (proved by
Harashita \cite{harashita:eodl}) that all EO strata which are
contained in the supersingular 
locus are disconnected.

The paper at hand (and Ekedahl and van der Geer's results
\cite{ekedahl-vdgeer}) clarifies the relationship between the KR
stratification and the EO stratification.
The precise relationship between our description of supersingular KR
strata and Hoeve's description of EO strata contained in $\mathcal S_g$
\cite{hoeve} is explained in \cite{goertz-hoeve}.

We were inspired at several places by Ekedahl and van der Geer
\cite{ekedahl-vdgeer}, and generalize some of their ideas. The relationship
between the situations considered here and in \cite{ekedahl-vdgeer}, 
respectively, is
explained in detail in Section \ref{sec:rel_to_EvdG}; see in particular
Proposition \ref{rel_to_EvdG}.  See also van der Geer
\cite{vdgeer:cycles}, where many of the results relevant for us are already
present.

Let us mention one point which facilitates the work of Ekedahl and van der
Geer in comparison with our situation (and allows them to go further than
we can). Namely, their main object of study, the bundle of flags inside the
Hodge filtration of the universal abelian scheme over $\mathcal A_g$, has a
natural compactification (because one can extend this bundle to the
Baily-Borel compactification of $\mathcal A_g$). No similarly well
understood compactification is available for $\mathcal A_I$.

Another way to compute the dimension of the supersingular locus (or of an
arbitrary Newton stratum) is to use the theory of affine Deligne-Lusztig
varieties. (In the non-supersingular case one would also have to study the
leaves in the Newton stratum in the sense of Oort.) However, even if one
ignores the problem that the results about affine Deligne-Lusztig varieties
which are currently available in the Iwahori case relate to the function
field case, the combinatorial complexity of the algorithm to compute the
dimension of such varieties given in \cite{GHKR2} is so high that computing
the dimension of $\mathcal S_I$ is entirely out of sight for $g>4$.

Let us survey the content of the individual sections. We start,
in Section 2, collecting the definitions and all the relevant results we need about
the $p$-rank stratification (on either $\mathcal A_g$ or $\mathcal A_I$),
the Ekedahl-Oort stratification (on $\mathcal A_g$) and the
Kottwitz-Rapoport stratification (on $\mathcal A_I$). Section 3 explains
that the image in $\mathcal A_g$ of each KR stratum is a union of EO
strata. This is a simple but extremely important fact. In Section 4, we
construct a morphism from $\mathcal A_I$ to the bundle of symplectic flags
in the first de Rham cohomology of the universal abelian scheme over
$\mathcal A_g$. This morphism is almost an embedding, but not quite: it is
finite, and a universal homeomorphism, but is inseparable. Using this map,
and a version of Raynaud's trick, we prove in Section 5 that all KR strata
are quasi-affine. The main result of Section 6 is that connected components
of KR strata are never closed in $\mathcal A_I$, unless they are
$0$-dimensional (and hence are components of the minimal KR stratum). This
is important for Section 7, where we show that non-superspecial KR strata
are connected. As a corollary we obtain (using Harashita's results about
the number of connected components of EO strata in the supersingular locus)
a proof that all supersingular KR strata are superspecial. Section 8
contains the computation of the dimension of the $p$-rank $0$ locus in
$\mathcal A_I$ (which gives us an upper bound on the dimension of the
supersingular locus in $\mathcal A_I$); this computation is independent of
the rest of the paper. At the end of this section, we derive consequences
about the supersingular locus in $\mathcal A_I$ from our results. Finally,
in Section 9 we discuss the relationship to the work of Ekedahl and
van der Geer \cite{ekedahl-vdgeer}.

% ========================================================================

\section{Preliminaries about EO and KR stratifications}

In this section we set up the notation and recall a number of results on
the $p$-rank stratification, the Ekedahl-Oort stratification, and the
Kottwitz-Rapoport stratification. No claim to originality is made.

\subsection{Notation}
\label{notation}

We fix a prime $p$ and an integer $g\ge 1$.  Let $k$ be
an algebraic closure of $\mathbb F_p$.  We denote by $\mathcal A_g$ the
moduli space of principally polarized abelian varieties of dimension $g$
over $k$.  We can either consider $\mathcal A_g$ as an algebraic stack, or
instead consider the moduli space of $g$-dimensional principally polarized
abelian varieties with a symplectic level $N$ structure (with respect to
some fixed primitive $N$-th root of unity), where $N\ge 3$ is an integer
coprime to $p$. In this way we obtain a quasi-projective scheme over $k$ of
dimension $g(g+1)/2$. As long as it does not matter which point of view we
choose (and mostly, it will not), we will not be very precise about this.

Inside $\mathcal A_g$ we have the supersingular locus, a closed subset of
$\mathcal A_g$, which we denote by $\mathcal S_g$. All its irreducible
components have dimension $[g^2/4]$, and it is known to be connected if $g
> 1$. See the book by Li and Oort \cite{li-oort}.

Furthermore, we denote by $\mathcal A_I$ the moduli space of tuples
\[
(A_0 \rightarrow A_1 \rightarrow \cdots \rightarrow A_g, \lambda_0,
\lambda_g),
\]
where all $A_i$ are abelian varieties of dimension $g$, the maps $A_i
\rightarrow A_{i+1}$ are isogenies of degree $p$, and $\lambda_0$,
$\lambda_g$ are principal polarizations on $A_0$ and $A_g$, respectively,
such that the pull-back of $\lambda_g$ to $A_0$ is $p\lambda_0$.
Here $I$ stands for \emph{Iwahori} type level structure at $p$. It can also
be seen as the index set $I = \{ 0, \dots, g\}$ indicating that we consider
full chains. Instead we could consider partial chains and would obtain more
general parahoric level structure at $p$. As above, whenever appropriate,
we could in addition consider a level structure outside $p$, in order to
obtain an honest scheme.

The dimension of $\mathcal A_I$ is $g(g+1)/2$, as well. Inside it, we have
the supersingular locus, i.~e.~the locus where one or equivalently all
abelian varieties in the chain are supersingular. Its dimension is not
known in general (but see \ref{dim_ss_locus}).

We have a natural projection map $\pi \colon \mathcal A_I \rightarrow
\mathcal A_g$ which maps $(A_\bullet, \lambda_0, \lambda_g)$ to $(A_0,
\lambda_0)$. This map is proper and surjective. It is not at all flat,
however: The fiber dimension jumps, and the behavior is quite complicated. 

The points in $\mathcal A_I$ are chains $A_0 \rightarrow \dots \rightarrow
A_g$ of isogenies of order $p$. Given such a chain , we can extend it ``by
duality'' to a chain $A_0 \rightarrow \dots \rightarrow A_{2g}$, i.~e.~for
$i=g, \dots, 2g$ we let $A_{2g-i} = A_i^\vee$ (and use $\lambda_g$ to
identify $A_g$ with $A_g^\vee$). Then $\lambda_0$ gives us an
identification of $A_{2g}$ with $A_0$. The isogeny $A_{2g-i} \rightarrow
A_{2g-i+1}$ is the dual isogeny for $A_{i-1}\rightarrow A_i$.

All of the above is explained in more detail in \cite{goertz-yu}, Section 2.

\subsection{Group-theoretic notation, I}
\label{group_theor_I}

In this section, we fix the notation related to the (extended affine) Weyl
group and the (affine) root system of the group $G= \GSp_{2g}$ of symplectic
similitudes.  For a more comprehensive account, we refer to
\cite{goertz-yu}, Subsections 2.1--2.3.

We use the Borel subgroup of upper triangular matrices and the maximal
torus $T$ of diagonal matrices with respect to the embedding $G
\subset \GL_{2g}$ induced by the alternating form $\psi$ such that for the
standard basis vectors $e_1, \dots, e_{2g}$,
\[
\psi(e_i, e_{2g-i+1}) = 1 = - \psi(e_{2g-i+1}, e_i),\ 1 \le i \le g,
\]
and all other pairs of standard basis vectors pair to $0$.
We denote by
$\wt W$ the extended affine Weyl group for $G$. We often regard it as a
subgroup of the extended affine Weyl group for $\GL_{2g}$ which we can
identify with the semidirect product $\mathbb Z^{2g} \rtimes S_{2g}$.
Inside $\wt W$, we have the affine Weyl group $W_a$, an (infinite) Coxeter
group generated by the simple affine reflections $s_0, \dots, s_g$.

The finite Weyl group $W$ of $G$ is the subgroup of $W_a$
generated by $s_1, \dots, s_g$.
For a translation element $\lambda \in X_*(T)$, we denote by $t^\lambda$
the corresponding element of $\wt W$.

For a subset $J \subseteq \{0, \dots, g\}$ we use the following \emph{somewhat
unusual} notation (as in \cite{goertz-yu}): $W_J$ denotes the subgroup
generated by the simple reflections $s_i$, $i \in \{0, \dots, g \}
\setminus J$. The case most relevant for us will be $J= \{ i, g-i\}$ for
some $0 \le i \le g/2$.

Since $W_a$ is a Coxeter group, the choice of generators $s_0, \dots, s_g$
gives rise to a length function $\ell$ and to the Bruhat order $\le$ on
$W_a$. Both can be extended to $\wt W$ in a natural way and these
extensions will be denoted by the same symbols.
The extended affine Weyl group is the semi-direct product $\wt W = W_a \rtimes
\Omega$, where $\Omega \cong \pi_1(\GSp_{2g}) \cong \mathbb Z$ is the
subgroup of $\wt W$ of elements of length $0$.

\subsection{The $p$-rank stratification}

Denote by $\mathcal A_g^{(i)}$ the locally closed subset where the $p$-rank
of the underlying abelian variety is $i$ (i.~e.~where $A[p](k) \cong
(\mathbb Z/p\mathbb Z)^i$). Likewise, let $\mathcal A_I^{(i)}
=\pi^{-1}(\mathcal A_g^{(i)})_{\rm red}$ be the $p$-rank $i$ locus inside
$\mathcal A_I$.  The $p$-rank $g$ locus is the \emph{ordinary locus}. The
closure of $\mathcal A_g^{(i)}$ is the union of all $\mathcal A_g^{(j)}$,
$j\le i$, but the closure of a ``stratum'' $\mathcal A_I^{(i)}$ in general
is not a union of $p$-rank strata, as can be seen already for $g=2$.

\begin{prop} \label{p_rk_0_proper}
The $p$-rank $0$ locus $\mathcal A_g^{(0)}$ is complete and equidimensional
of dimension $g(g-1)/2$. The locus $\mathcal A_I^{(0)}$ is
proper over $k$.
\end{prop}

\begin{proof}
See the papers by Koblitz \cite{koblitz:thesis}, Theorem~7, and Oort
\cite{oort74}, Theorem~1.1 (a) for the assertions about $\mathcal A_g^{(0)}$.
The properness of $\mathcal A_I^{(0)}$ follows immediately.
\qed
\end{proof}

More generally, Koblitz shows that the $p$-rank $i$ locus $\mathcal
A_g^{(i)}$ inside $\mathcal A_g$ is equidimensional of dimension $g(g-1)/2 +
i$. Furthermore, it is known that $\mathcal A_g^{(0)}$ has the minimal
possible codimension which a complete subvariety of $\mathcal A_g$ can
have; see \cite{vdgeer:cycles}, Corollary~2.7.  We will prove later that $\dim
\mathcal A_I^{(0)} = [g^2/2]$ (see Theorem~\ref{dim_prk0}).

\subsection{Results about EO strata}

The Ekedahl-Oort stratification was defined in \cite{oort01} (where it is
called the \emph{canonical stratification}). It is the stratification on
$\mathcal A_g$ given by the isomorphism type of the $p$-torsion points
$A[p]$. We denote by $\ES$ the set of strata. For $w\in \ES$, we
denote by $EO_w$ the corresponding stratum. Each $EO_w$ is locally closed
in $\mathcal A_g$, and the closure of a stratum is a union of strata.

In the literature, there are three ways to describe the set $\ES$. Let us
briefly discuss how they are related. In  \cite{oort01}, Oort classifies EO
strata by \emph{elementary sequences}, i.~e.~maps $\varphi\colon \{ 0,
\dots, g \} \rightarrow \{ 0, \dots, g \}$ such that $\varphi(0) = 0$ and
$\varphi(i) \le \varphi(i+1) \le \varphi(i)+1$. (Ekedahl and) van der Geer
(\cite{vdgeer:cycles}, see also \cite{ekedahl-vdgeer}, 2.2) use the
set of minimal 
length representatives in $W$ for the cosets $W/S_g$. This is the
description we will use, too, so we make the following definition.

\begin{defn}
Let $S_g \subset W$ denote the subgroup generated by $s_1, \dots, s_{g-1}$
(this subgroup is isomorphic to the symmetric group on $g$ letters, as the
notation indicates).
An element $w\in W$ is called \emph{final}, if it is the unique element of
minimal length in the coset $wS_g$. We denote the set of final elements by
$W_{\rm final}$.
\end{defn}

Explicitly, an element $w\in W$ is in $W_{\rm final}$ if and only if
\[
w(1) < w(2) < \cdots < w(g).
\]
The bijection $W_{\rm final} \isomarrow \ES$ is given by
\[
w \mapsto \nu_w,\qquad \nu_w(i) = i - \# \{ a \in \{1,\dots, g\};\
w(a) \le i \}.
\]

Finally, Moonen and Wedhorn \cite{moonen1}, \cite{moonen-wedhorn}, see also
\cite{wedhorn:goettingen}, 5.2, 6.2, describe the set of EO strata as the
quotient $S_g\backslash W$, or equivalently as the set of minimal length
representatives in $W$ for the cosets in $S_g\backslash W$. This
description is related to the description by the set $W_{\rm final}$ by the
map $w \mapsto w^{-1}$. Correspondingly, using the description of Moonen
and Wedhorn, an element $w\in W$ corresponds to the elementary sequence
\[
w \mapsto \varphi_w,\qquad \varphi_w(i) = \# \{ a \in \{1, \dots, i\};\ w(a)
> g \}.
\]

\begin{prop} \label{eo0}

\begin{enumerate}
\item
There is a natural bijection between the set $\ES$ of EO strata and the set
$W_{\rm final}$ of final elements in $W$.
\item
Every EO stratum is quasi-affine.
\item
Let $w\in W_{\rm final}$, and let
$\varphi$ be the corresponding elementary sequence. Then $EO_w$ is
equidimensional of dimension
\[
 \dim EO_w = \ell(w) = \sum_{i=1}^g \varphi(g).
\]
\end{enumerate}
\end{prop}

\begin{proof}
For~(1), see \cite{moonen1}, Theorem~4.7. Part (2), as
well as the dimension formula in~(3), is proved in \cite{oort01},
Theorem~1.2.  For the equidimensionality in~(3), see \cite{moonen2},
Corollary~3.1.6. See also \cite{ekedahl-vdgeer}. 
\qed
\end{proof}

There is a unique $0$-dimensional EO stratum; it is precisely the locus of
superspecial abelian varieties in $\mathcal A_g$. There is a unique
one-dimensional EO stratum. See \cite{oort01}, \S 1. Finally, there is also
a unique open (and hence dense) stratum. The open stratum is equal to the
ordinary locus $\mathcal A_g^{(g)}$.

\begin{prop} \label{eo1}
\hspace*{1cm}

\begin{enumerate}
\item
The EO stratification is a refinement of the $p$-rank stratification.
\item
For $w \in  W_{\rm final}$, the $p$-rank on $EO_w$ is
\[
\# \{ i\in \{1, \dots, g \};\ w(i) = g+i  \},
\]
where we consider $w$ as an element of the symmetric group $S_{2g}$ by the
natural embedding $W_{\rm final} \subset W \subset S_{2g}$.
\end{enumerate}
\end{prop}

\begin{proof}
Part (1) is obvious from the definition of the EO stratification. For part
(2), see \cite{ekedahl-vdgeer}, Lemma 4.9 (i).
\qed
\end{proof}

\begin{prop} \label{properties_EO}
Let $w \in W_{\rm final}$.
\begin{enumerate}
\item
The stratum $EO_w$ is contained in $\mathcal
S_g$ if and only if $w(i)=i$ for $i = 1, \dots, g-[g/2]$ (where we consider
$w$ as an element of the symmetric group $S_{2g}$ as in
Proposition~\ref{eo1} (2)). 
\item
If $EO_w$ is contained in $\mathcal S_g$ and $N\ge 4$, 
% and $p$ and/or $N$ large (MAKE PRECISE!), 
then $EO_w$ is not connected.
\item
If $EO_w$ is not contained in $\mathcal S_g$, then $EO_w$ is
irreducible.
\item
Every EO stratum is smooth. In particular, every connected component of
$EO_w$ is irreducible.
\item
Let $X\subseteq EO_w$ be a connected component. Then the closure of
$X$ meets the superspecial locus of $\mathcal A_g$.
\end{enumerate}
\end{prop}

\begin{proof}
Harashita (\cite{harashita:eodl}, Proposition~5.2) proves the following result
which is due to Oort: The EO stratum associated to an elementary
sequence $\varphi$ is contained in $\mathcal S_g$ if and only if
$\varphi(g-[g/2]) = 0$. For the corresponding element $w\in W_{\rm
final}$, this means $\{ 1, \dots, g- [g/2] \} \subseteq w(\{1,\dots, g\})$,
and because $w(1) < \cdots < w(g)$, this is equivalent to the condition
stated in (1).

% et $\varphi$ be an elementary sequence such that $EO_\varphi
% \subseteq \mathcal S_g$. Then we know by part (1) that there exists a
% unique integer $c$, $0 \le c \le [g/2]$ such that $\varphi(g-c) = 0$,
% $\varphi(g-c+1)=1$.
% Theorem 6.17 in \cite{harashita1} states that in this situation the
% number of irreducible 
% components of $EO_\varphi$ \emph{in the coarse moduli space $\mathcal
% A_{g,1,1}$} is $H_{g,c}$, a class number associated to a certain unitary
% group. (EXTEND)

Part (3) was proved by Ekedahl and van der Geer \cite{ekedahl-vdgeer},
Theorem~11.5. In fact, in order to align this with the condition given
in (1), 
it is easier to start from their Theorem 11.4 and Lemma 11.3. By (1),
if $EO_w$ is not contained in $\mathcal S_g$, then there exists 
$i \in \{ 1, \dots, g-[g/2] \}$ with $w(i) \ne i$. 
Because $w$ is final, this implies that all
simple reflections $s_{g-[g/2]}, \dots, s_g$ are less than $w$ in the
Bruhat order (i.~e.~occur in any reduced expression of $w$). The
results in \cite{ekedahl-vdgeer} then imply that $EO_w$ is irreducible.

The smoothness claimed in Part (4) can be extracted from Oort's paper
\cite{oort01}. It was shown in a different way by Wedhorn in
\cite{wedhorn:texel}; see Corollary~3.5 and Theorem~6.4 (at least if
$p>2$). It 
also follows from the results in \cite{ekedahl-vdgeer} (in particular
Corollary~8.4 (iii)). The second statement follows immediately.

To prove (2), by the mass formulas (cf.~\cite{goertz-yu}, Proposition~3.9)
and Theorem 6.17 in Harashita \cite{harashita:eodl}, the number of
the irreducible components of $EO_w$ has the form $|\Sp_{2g}(\Z/N\Z)|
L_w(p)$ for any prime $p$ and prime-to-$p$ positive integer $N\ge 3$,
where $L_w(X)$ is a polynomial over $\Q$ which is positive and
increasing for positive numbers. As the formula gives a natural number
for $N=3$ and $p=2$, it is enough to show that for $N\ge 4$ one has
$|\Sp_{2g}(\Z/N\Z)|\ge 2 |\Sp_{2g}(\Z/3\Z)|$. The latter is easily
verified for $N=4$ and $N=5$ and holds for general $N\ge 4$. 

Finally, we prove part (5). If $g=1$, then there is nothing to prove.  By
\cite{oort01}, Theorem~1.3, the closure of $X$ contains a component of the
closure of the unique one-dimensional EO stratum, which is contained in the
$p$-rank $0$ locus as soon as $g>1$. Since every such component
is quasi-affine by \cite{oort01}, Theorem~4.1, and of dimension $1$,
it is not proper and hence cannot be closed in the $p$-rank $0$ locus.
Therefore it 
contains a superspecial point. Alternatively, (5) also follows from part
(3), Proposition~\ref{eo0} (2) and the fact that $\mathcal A_g^{(0)}$
is proper. 
\qed
\end{proof}

\subsection{Results about KR strata}

On $\mathcal A_I$, we have the \emph{Kottwitz-Rapoport stratification}
which is given by the relative position of the chains $H^1_{DR}(A_i)$ and
$\omega(A_i)$, the latter denoting the Hodge filtration inside
$H^1_{DR}(A_i)$. This relative position is an element in $\wt W$; the set
of possible relative positions is called the $\mu$-admissible set and
denoted by ${\rm Adm}(\mu)$. It is a finite set, closed under the Bruhat
order, and its maximal elements with respect to the Bruhat order are the
translation elements $t^{w(\mu)}$, where $w\in W$ and $\mu = (1^{(g)},
0^{(g)}) \in \mathbb Z^{2g}$.  There exists a unique element $\tau \in \wt
W$ of length $0$ such that ${\rm Adm}(\mu) \subset W_a\tau$ (and $\tau \in
{\rm Adm}(\mu)$). We denote the finite part of $\tau$ by $w_\emptyset$. This is at the same time the longest final element.  For $x
\in {\rm Adm}(\mu)$, we denote the KR stratum associated with $x$ by $\mathcal
A_x$.

The KR stratification was first considered by Genestier and Ng\^{o},
\cite{ngo-genestier:alcoves}. See also \cite{haines:clay} for a
detailed exposition. We refer to Section 3 below and to
\cite{goertz-yu}, Section 2, for further details. 

More or less by definition the KR stratification is smoothly equivalent to
the so-called local model $M^{\rm loc}$ (for $G$ and $\mu$) with its natural
stratification given by orbits under the action of the Iwahori group
scheme.
In \cite{goertz:symplectic} it is shown that the special fiber of the local
model is isomorphic to the union
\[
\bigcup_{x \in {\rm Adm}(\mu)} C_x
\]
of Schubert varieties in the affine flag variety for $G$. In particular we
obtain

\begin{prop}
\hspace*{1cm}

\begin{enumerate}
\item
For $x \in {\rm Adm}(\mu)$, the stratum $\mathcal A_x$ is smooth of
dimension $\ell(x)$.
\item
The closure relations between KR strata are given by the Bruhat order:
We have $\mathcal A_x \subseteq \overline{\mathcal A}_y$ if and only if $x
\le y$.
\item
All closures of KR strata are normal and equidimensional.
\end{enumerate}
\end{prop}

\begin{proof}
We have a ``local model diagram''
\[
\mathcal A_I \leftarrow \widetilde{\mathcal A}_I \rightarrow M^{\rm loc},
\]
where both morphisms are smooth of the same relative dimension;
see~\cite{rapoport:guide} or \cite{haines:clay}. By definition of the
KR stratification on $\mathcal A_I$, this diagram restricts to a
corresponding diagram where on the right hand side we have the
corresponding Schubert cell (or its closure, resp.), and on the left
hand side, we have the KR stratum (or its closure, resp.), and where
again both arrows are smooth. The assertions in (1), (2) and (3) then
follow from the same statements for Schubert cells, as soon as we know that all
KR strata are non-empty.

By the above, it is enough to show that the stratum $\mathcal A_\tau$ is
non-empty. But it is not hard to write down chains $(A_i)_i$ lying in $A_\tau$
explicitly, starting from a product $A_0$ of supersingular elliptic curves; see
Genestier \cite{genestier:ressing}, Proposition~1.3.2.
\qed
\end{proof}

The relationship to the $p$-rank stratification was analyzed by Ng\^o and
Genestier:

\begin{prop} \label{ngo-gen}
The KR stratification is a refinement of the stratification by $p$-rank.
The $p$-rank on the stratum $\mathcal A_x$, $x \in {\rm Adm}(\mu)$, is
given by $\#\{ i \in \{ 1, \dots, g \};\ w(i) = i \}$, where $x = t^\lambda
w$, $w \in W$, and where $w$ is considered as an element of $S_{2g} \supset
W$.
\end{prop}

\begin{proof}
See \cite{ngo-genestier:alcoves}.
\qed
\end{proof}

% In particular, we have
% 
% \begin{cor}
% The ordinary locus $\mathcal A_I^{(g)}$ is equal to the union
% \[
% \bigcup_{\lambda\in W\mu} \mathcal A_{t^\lambda}
% \]
% of maximal KR strata. In particular, it is open and dense in $\mathcal
% A_I$, and smooth over $k$.
% \end{cor}
% 
% (In fact, if we consider the analogous moduli problem over $\mathbb Z_p$,
% it is still true that the points of $\mathcal A_I^{(g)}$ are smooth points
% of this space.)

\begin{figure}[h]
\includegraphics[width=8cm]{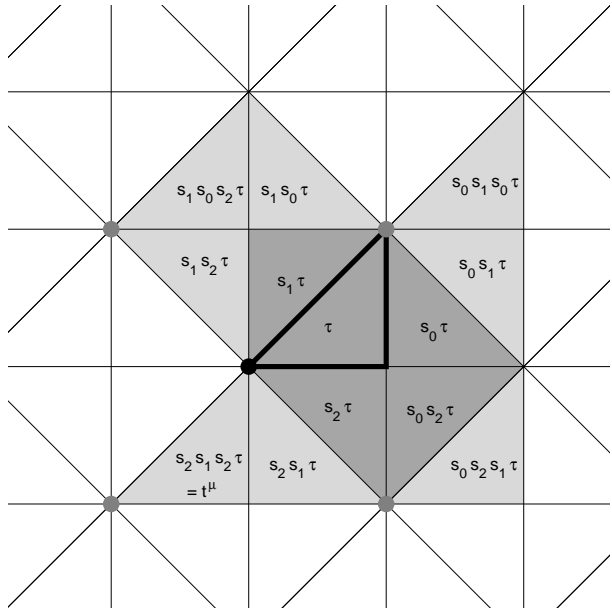}
\caption{
    The $\mu$-admissible set for $\GSp_4$.    }
\end{figure}

\begin{eg}
The figure shows the 13 alcoves in the $\mu$-admissible set for $g=2$. The
black dot marks the origin; the gray dots show the translations in the
$W$-orbit of $\mu$. The dark gray alcoves are those of $p$-rank $0$. Note
that the picture is very similar to the one in \cite{haines:clay}, Fig.~2.
The reduced expressions of the admissible elements coincide, but in our
setup, the base alcove $\tau$ lies in the anti-dominant chamber.
\end{eg}

% ========================================================================

\section{The image of KR strata in $\mathcal A_g$}

\subsection{}
Let $\pi: \mathcal A_I\to \mathcal A_g$ be the natural
projection. We will show below that each KR stratum is mapped to a union of
EO strata under $\pi$, a fact which will be of great importance in the
following sections. We start with a lemma which says in particular that the
KR stratum of a point $(A_i)_i\in \mathcal A_I(k)$, which a priori depends
on the \emph{chains} $(H^1_{DR}(A_i))_i$, $(\omega(A_i))_i$, is in fact
already determined by the \emph{flags} inside $H^1_{DR}(A_0)$ which arise
as the images of these chains.  Let us make this more precise by
introducing some more notation. See also \cite{goertz-yu}, Section 2.

\emph{Relative position of chains.} Given a point $(A_i)_i\in \mathcal
A_I(k)$, we have the chains $(H^1_{DR}(A_i))_i$ of $2g$-dimensional
$k$-vector spaces, and $(\omega(A_i))_i$ of $g$-dimensional subspaces. By
using duality, we get chains indexed by $i=0, \dots, 2g$. We will express
their relative position in terms of a tuple of vectors in $\mathbb Z^{2g}$.
For $v\in \mathbb Z^{2g}$, we denote by $v(1),\dots, v(2g)$ the entries of
$v$.  Denote by $(\omega_i)_{i=0,\dots,2g}$ the tuple with $\omega_i =
((-1)^{(i)}, 0^{(2g-i)})\in\mathbb Z^{2g}$. The relative position of
the chains 
$(H^1_{DR}(A_i))_i$ and $(\omega(A_i))_i$ is the tuple $(x_i)_{i=0, \dots,
2g}$, $x_i\in \mathbb Z^{2g}$, such that there exist $k$-bases $e^i_1,
\dots, 
e^i_{2g}$ of $H^1_{DR}(A_i)$, $i=0, \dots, 2g$, with the properties
\begin{enumerate}
\item
With respect to these bases, the map $H^1_{DR}(A_i) \rightarrow
H^1_{DR}(A_{i-1})$ is given by the diagonal matrix 
${\rm diag}(1,\dots, 1, 0, 1,\dots, 1)$, the $0$ being in the
$2g-i+1$-th place. 
\item
For all $i$, $j$, we have $(x_i-\omega_i)(j) \in \{0, 1\}$.
\item
The vectors $e^i_j$ with $(x_i-\omega_i)(j)=0$ are a basis of the
$g$-dimensional subspace $\omega(A_i)\subset H^1_{DR}(A_i)$. 
\end{enumerate}
One checks that the tuple $(x_i)_i$ does not depend on the choice of basis.
The strange-looking normalization implies that $(x_i)_i$ is an
\emph{extended alcove} in the sense of \cite{kottwitz-rapoport:alcoves},
using the normalization of \cite{goertz-yu} Section 2. The extended
affine Weyl group $\wt W$ acts simply transitively on the set of
extended alcoves. If we write $x\in\wt W$ as $x = t^\lambda w$, $w\in
W$, then $x$ acts on an alcove $(y_i)_i$ by $x.(y_i)_i =
(\lambda+w(y_i))_i$. It is clear by construction 
that all alcoves which arise as relative positions in the above
situation are permissible (or equivalently, admissible);
see~\cite{kottwitz-rapoport:alcoves} or \cite{goertz-yu}, Definition~2.4. 

\emph{Relative position of extended flags.} Consider a $2g$-dimensional
symplectic vector space $(V, \langle \cdot,\cdot \rangle)$ over $k$ with a
complete self-dual flag $0=V_0\subsetneq V_1 \subsetneq \cdots \subsetneq
V_{2g-1} \subsetneq V$ and an ``extended flag'', by which we mean a system
$0= F_0 \subseteq F_1 \subseteq \cdots \subseteq F_{2g}$ of $V$, such that
$F_{2g}$ is a maximal totally isotropic subspace (in particular $\dim
F_{2g} = g$), $F_i \subseteq V_i$, and $\dim F_i/F_{i-1} \in \{0,1\}$ for
all $i$. The relative position of such a system is the tuple $(\tilde
x_i)_{i=0, \dots, 2g}$, $\tilde x_i\in \mathbb Z^{2g}$, such that there
exists a $k$-basis $e_1, \dots, e_{2g}$ of $V$ with the properties
\begin{enumerate}
\item
For every $i$, $V_i$ is spanned by $e_1, \dots, e_i$.
\item
For all $i$, $j$, we have $\tilde x_i(j) \in \{0, -1\}$.
\item
For all $i$, $F_i$ is the subspace of $V$ spanned by those $e_j$ with
$\tilde x_i(j) = -1$.
\end{enumerate}
We obtain a system like this in the above setting by defining
$V=H^1_{DR}(A_{0})$, $V_i = \im (H^1_{DR}(A_{2g-i})\rightarrow
H^1_{DR}(A_0))$, $F_i = \im (\omega(A_{2g-i})\rightarrow H^1_{DR}(A_0))$.
From the relative position $(x_i)_i$ of chains we obtain the relative
position of extended flags by setting $\tilde x_i(j) = \min(x_i(j),0)$.

\emph{Relative position of flags.} As before, fix a $2g$-dimensional
symplectic vector space $(V, \langle \cdot,\cdot \rangle)$ over $k$ with a
complete self-dual flag $0=V_0\subsetneq V_1 \subsetneq \cdots \subsetneq
V_{2g-1} \subsetneq V$. Now fix a second complete self-dual flag, or what
amounts to the same, a partial flag $0 = F'_0 \subsetneq F'_1 \subsetneq
\cdots \subsetneq F'_g$ with $\dim F'_i = i$, $F'_g$ maximal totally
isotropic.  The relative position of such a system is the tuple
$(x'_i)_{i=0, \dots, g}$, $x'_i\in \mathbb Z^{2g}$, such that there
exists a $k$-basis $e_1, \dots, e_{2g}$ of $V$ with the properties
\begin{enumerate}
\item
For every $i$, $V_i$ is spanned by $e_1, \dots, e_i$.
\item
For all $i$, $j$, we have $x'_i(j) \in \{0, -1\}$.
\item
For all $i$, $F_i$ is the subspace of $V$ spanned by those $e_j$ with
$x'_i(j) = -1$.
\end{enumerate}
Of course, we can also express the relative position as an element of the
Weyl group $W$: it is the unique element $v\in W$ such that $v\omega_i =
x'_i$ for $i=0,\dots, g$.
Given an extended flag as above, and the corresponding relative position
$(\tilde x_i)_i$, we obtain a usual flag by forgetting the steps where we
have equality, and obtain the relative position of flags by analogously
forgetting multiple occurrences of vectors among the $\tilde x_i$'s.

\begin{lm} \label{lm:rel_pos_from_flag}
Let $(x_i)_{i=0,\dots, 2g}\in{\rm Adm}(\mu)$ be an admissible extended
alcove.  Let $x \in \wt W$ with $x(\omega_i)_i = (x_i)_i$, and write $x =
t^{\lambda(x)} w = w t^{\rho(x)}$ for $w\in W$, $\lambda(x),\rho(x)\in
X_*(T)$.  Define $(\tilde x_i)_{i=0, \dots, 2g}$ by $\tilde x_i(j) =
\min(x_i(j),0)$.
\begin{enumerate}
\item
The extended alcove $(x_i)_i$ is determined by the datum $(\tilde
x_i)_{i=0, \dots, 2g}$. We can recover $w, \lambda, \rho$ as follows:
$\lambda(x) = \tilde x_{2g} + (1,\dots, 1)$, $\rho(x)$ is given by
\[
\rho(x)(i) = \left\{ \begin{array}{ll}
1 & \text{ if } \tilde x_{i-1} = \tilde x_i\\
0 & \text{ if } \tilde x_{i-1} \neq \tilde x_i 
\end{array}\right.,\qquad i=1,\dots, 2g.
\]
Finally $w\in W \subset S_{2g}$ is the unique permutation such that for
every $i$ with $\tilde x_{i-1} \neq \tilde x_i$ the difference $\tilde
x_i-\tilde x_{i-1}$ has its unique $-1$ in the $w(i)$-th place.
\item
Now define $(x'_i)_{i=0,\dots, g}$ by letting $x'_i$ be the unique vector
among the $\tilde x_*$ with $\sum_j x'_i(j) = -i$ (this vector may occur
several times among the $\tilde x_*$). Let $v\in W$ be the unique element
such that $v(\omega_i) = x'_i$, $i=0,\dots, g$. Denote by $w_0\in W$ the
longest element, and let $w_{\rho(x)}$ be the unique element of minimal length
with $w_{\rho(x)}(\mu) = w_0(\rho(x))$. (Here the minimality is equivalent to
requiring that $w_{\rho(x)}$ be final.) Then $v = ww_{\rho(x)}$.
\item
Continuing with the above notation, if $\rho(x) = \mu$, then $x = wt^\mu =
v\tau$.
\end{enumerate}
\end{lm}

\begin{proof}
We start with (1). First, $\lambda(x) = x_0 = x_{2g} + (1,\dots, 1) =
\tilde x_{2g} + (1,\dots, 1)$. Furthermore, for all $i$, $x_i - x_{i-1} =
(0,\dots, 0, -1, 0,\dots, 0)$, the $-1$ being in the $w(i)$-th place. This
implies the desired formulas for $\rho(x)$ and $w$. The description in (2)
is also easy to check. For (3), use that with the above notation, $\tau =
w_\mu t^\mu$.
\qed
\end{proof}

The lemma can also be seen in the spirit of the numerical characterization
of KR strata as in \cite{goertz-yu}.   Now recall that $W_{\rm final}$ (or
equivalently, the set of elementary sequences)
parametrizes the set of EO strata on $\mathcal A_g$. For $w\in W_{\rm
final}$,  we denote by $EO_w$ the corresponding EO stratum.

\begin{prop} \label{KR_maps_to_EO}
   If $\pi(\mathcal A_{I,x})\cap EO_w\not=\emptyset$ for some
  $x\in {\rm Adm}(\mu)$ and $w\in W_{\rm final}$, then $\pi(\mathcal
  A_{I,x})\supseteq EO_w$.
\end{prop}
\begin{proof}
  An Iwahori level structure on an object $(A,\lambda,\eta)$ in $\mathcal
  A_g$ is the same as a flag $H_\bullet$ of finite flat subgroup schemes of
  $A[p]$ satisfying certain properties and 
  the Kottwitz-Rapport invariant is determined by this
  structure. In particular if two objects $(A,\lambda,\eta,H_\bullet)$ and
  $(A',\lambda',\eta',H_\bullet')$ in $\mathcal A_I$ have the property
  that $(A[p], e_\lambda, H_\bullet)\simeq (A'[p], e'_\lambda, H'_\bullet)$,
  where $e_\lambda$ denotes the pairing on $A[p]\times A[p]$ defined by $\lambda$,
  then they lie in the same KR stratum. This follows from Lemma
  \ref{lm:rel_pos_from_flag} (2).

  Let $(A,\lambda,\eta, H_\bullet)\in \mathcal A_{I,x}$ be a point
  such that $(A,\lambda,\eta)\in EO_w$. Let
  $(A',\lambda',\eta')$ be a point in $EO_w$. One has an
  isomorphism $\alpha:(A'[p],e_{\lambda'})\simeq
  (A[p],e_{\lambda})$. Set $H_i':=\alpha^{-1}(H_i)$. We get a point
  $(A',\lambda',\eta',H'_\bullet)$ in $ \mathcal A_{I}$ such that there is
  an isomorphism
\[ \alpha: (A'[p],e_{\lambda'}, H'_\bullet)\simeq
(A[p],e_{\lambda},H_\bullet). \]
  Therefore, $(A',\lambda',\eta',H'_\bullet)$ lies in $\mathcal A_{I,x}$
  and hence $\pi(\mathcal A_{I,x})\supset EO_w$.
\qed
\end{proof}

Note that in \cite{ekedahl-vdgeer}, Proposition~9.6, the case of KR strata
for $w\tau$, $w\in W$, is treated.

For $x \in {\rm Adm}(\mu)$, put 
\[ {\bf ES}(x):=\{ w\in W_{\rm final} \,; \, \pi(\mathcal
A_{I,x})\cap EO_w\not=\emptyset\, \}. \]

\begin{cor}
  For $x \in {\rm Adm}(\mu)$, one has
\[  \pi(\mathcal A_{I,x})=\coprod_{w\in {\bf ES}(x)}
EO_w. \]
\end{cor}

Let $w\in W_{\rm final}$.  Then $w\tau$ is an element of the admissible
set: We have $w \le w_\emptyset$, because $w_\emptyset$ is the longest
final element, and therefore $w\tau \le w_\emptyset\tau = t^\mu$. 
Ekedahl and van der Geer \cite{ekedahl-vdgeer} show that the image 
of $\mathcal A_{w\tau}$ in $\mathcal A_g$ is the EO stratum $EO_w$, in
other words ${\bf ES}(w\tau) = \{ w \}$. Furthermore, they show that the
restriction of $\pi$ induces a finite \'etale surjective map $\mathcal
A_{w\tau} \rightarrow EO_w$ (although in  \cite{ekedahl-vdgeer} this
is phrased in a different way; see Section~\ref{sec:rel_to_EvdG}).

It seems to be an interesting problem to describe the sets ${\bf ES}(x)$,
where $x \in {\rm Adm}(\mu)$. See \cite{ekedahl-vdgeer},
Section 13, for some results in the case when $x \in W\tau$.
% See also section \ref{tau_p} for a variant of this question.

\begin{eg}\label{33}
In the case $g=2$ (see the example in the previous section), we have
4 EO strata, which can be described as the superspecial locus, the
supersingular locus minus the superspecial points, the $p$-rank $1$ locus and
the $p$-rank $2$ locus, and which correspond to the following final elements
\[
w_{\rm ssp} = {\rm id},\ w_{\rm ssi} = s_2,\ w_1 = s_1 s_2,\ 
w_2 = s_2 s_1 s_2.
\]
In this case, it is not hard to write down all the sets ${\bf ES}(x)$. As
an abbreviation, we write $s_{120} := s_1 s_2 s_0$ etc.
\[
\begin{array}{l|l|l}
x & {\bf ES}(x) & p-{\rm rank} \\\hline
\tau, s_1\tau & \{ w_{\rm ssp} \} & 0 \\
s_{0}\tau, s_{2}\tau & \{ w_{\rm ssi} \} & 0 \\
s_{02}\tau & \{ w_{\rm ssp}, w_{\rm ssi} \} & 0 \\
s_{01}\tau, s_{21}\tau, s_{10}\tau, s_{12}\tau & \{ w_1 \} & 1 \\
s_{120}\tau, s_{010}\tau, s_{212}\tau, s_{021}\tau & \{ w_2 \} & 2
\end{array}
\]
\end{eg}

% ========================================================================

\section{A flag bundle over $\mathcal A_g$}

\subsection{Definition of ${\rm Flag}^\perp(\mathbb H)$}

Let $f\colon A^{\rm univ} \rightarrow \mathcal A_g$ denote the universal
abelian scheme. We write $\mathbb H:=H^1_{DR}(A^{\rm univ}/\mathcal A_g):=
R^1 f_* \Omega^\bullet$. This is a locally free $\mathcal O_{\mathcal
A_g}$-module of rank $2g$. Since $A^{\rm univ}$ is equipped with a
principal polarization, $\mathbb H$ carries a non-degenerate alternating
pairing. Furthermore, Frobenius and Verschiebung on $A^{\rm univ}$ induce
operators $F \colon \mathbb H^{(p)} \rightarrow \mathbb H$, $V\colon
\mathbb H \rightarrow \mathbb H^{(p)}$ (where $\mathbb H^{(p)}$ denotes the
pull-back of $\mathbb H$ under Frobenius.

\begin{defn}
Denote by $\mathbf p \colon {\rm Flag}^\perp(\mathbb H)\rightarrow \mathcal A_g$
the variety of symplectic flags in $\mathbb H$.
\end{defn}

\subsection{The map $\mathcal A_I \rightarrow {\rm Flag}^\perp(\mathbb H)$}

The flag bundle ${\rm Flag}^\perp(\mathbb H)$ is closely related to the
moduli space $\mathcal A_I$ with Iwahori level structure. There are several
ways to establish a relationship (see Section \ref{sec:rel_to_EvdG}
and \cite{ekedahl-vdgeer} for 
a slightly different one). Here we use the following naive approach.

\begin{defn}
Denote by $\iota$ the morphism $\mathcal A_I \rightarrow {\rm Flag}^\perp(\mathbb H)$
defined as follows. For $(A_i)_i \in \mathcal A_I(S)$, the image point
in ${\rm Flag}^\perp(\mathbb H)(S)$ is given by the pair $(A_0,
F_\bullet)$, where $F_\bullet$ is the flag 
\[
0 =
\alpha(H^1_{DR}(A_{2g})) \subset \alpha(H^1_{DR}(A_{2g-1})) \subset \cdots
\subset \alpha(H^1_{DR}(A_{1})) \subset \alpha(H^1_{DR}(A_{0})) = \mathbb
H.
\]
Here for each $i$, $\alpha$ denotes the natural map $H^1_{DR}(A_{i})
\rightarrow H^1_{DR}(A_{0})$.
\end{defn}

\begin{lm}
The morphism $\iota$ is universally injective and finite. In particular, it
is a homeomorphism of $\mathcal A_I$ onto a closed subset of ${\rm
Flag}^\perp(\mathbb H)$.
\end{lm}

\begin{proof}
For an algebraically closed field $K$, the morphism induces an injection on
$K$-valued points: if $A_0$ is given, then we can recover the chain of
\dieu modules, and hence the chain $(A_i)_i$ from the flag $F_\bullet$.
Since $\mathcal A_I$ is proper over $\mathcal A_g$, it is clear that
$\iota$ is proper; being proper and quasi-finite, it is finite.
\qed
\end{proof}

Note however that the induced maps on tangent spaces is not injective. In
particular, $\iota$ is not a closed immersion. It is easy to describe the
image of $\iota$ (in terms of points over an algebraically closed field):
We just must ensure that the flag gives rise to a chain of \dieu modules,
i.~e.~that all its members are stable under Frobenius and Verschiebung.
This is not to say, however, that the geometry of this locus would be easy
to understand.

\begin{eg}
Let us discuss the case $g=1$. This is not a typical case because it is too
simple in some respects, but it is nevertheless instructive. The Iwahori
type moduli space $\mathcal A_I$ can be described as the union of 2 copies
of the smooth curve $\mathcal A_g$, glued at the finitely many
supersingular points, so that they intersect transversely in these points.
More precisely, we have closed subschemes $\mathcal A$, $\mathcal A'$ of
$\mathcal A_I$, both isomorphic to $\mathcal A_g$, which we can describe as
\begin{eqnarray*}
\mathcal A & = & \{ (\xymatrix{E \ar[r]^F & E^{(p)}}) \} \\
\mathcal A' & = & \{ (\xymatrix{E^{(p)} \ar[r]^V & E}) \}.
\end{eqnarray*}
The intersection of $\mathcal A$ and $\mathcal A'$ inside $\mathcal A_I$
can be identified with the discrete set of supersingular points of either
$\mathcal A$ or $\mathcal A'$.

The projection $\pi\colon \mathcal A_I \rightarrow \mathcal A_g$ is the identity map
when restricted to $\mathcal A$ (and hence induces isomorphisms on the
tangent spaces of non-supersingular points), and is the Frobenius morphism when
restricted to $\mathcal A'$ (and hence induces the zero map on the tangent
spaces of non-supersingular points). If $x\in \mathcal A_I(k)$ is a
supersingular point, we can identify its tangent space with the product
$T_x\mathcal A \times T_x\mathcal A'$, and $\ker(d\pi\colon T_x\mathcal A
\times T_x\mathcal A' \rightarrow T_{\pi(x)}\mathcal A_g) = T_x\mathcal A'$.

Now let us investigate the morphism $\iota$. Let $x \in \mathcal A_I(k)$.
We can write $T_{\iota(x)}{\rm Flag}^\perp\mathbb H$ as a product of
$T_{\pi(x)}\mathcal A_g$ and the tangent space $T_f$ of $\iota(x)$ within the
fiber over $\mathcal A_g$. Then the map
\[
\xymatrix{
    T_x\mathcal A_I \ar[r]^(.35){d\iota_x} & T_{\iota(x)}{\rm Flag}^\perp\mathbb
    H \ar[r]^(.65){\rm proj.} & T_f
}
\]
is the zero map. The analogous statement is true for all $g$, because for
an abelian variety $A$ over $k$, all lifts $\tilde A$ of $A$ to
$k[\varepsilon]/\varepsilon^2$ give rise to ``the same'' $H^1_{DR}(\tilde
A)$---the crystal of $A$ evaluated at $k[\varepsilon]/\varepsilon^2$.

In this case $\iota$ induces a closed immersion $\mathcal A \rightarrow
{\rm Flag}^\perp\mathbb H$, but it induces the zero map on the tangent
spaces of the non-supersingular points in $\mathcal A'$. In particular,
even in the case $g=1$, $\iota$ is not a closed immersion.
\end{eg}

% ========================================================================

\section{KR strata are quasi-affine}

In this section we prove that all KR strata are quasi-affine. We follow the
method used in \cite{ekedahl-vdgeer} Lemma 6.2, based on ``Raynaud's
trick''.

\subsection{Ample line bundles on $\mathcal A_I$}

We use the following lemmas to produce an ample sheaf on $\mathcal A_I$.

\begin{lm}\label{qa-lem1}
Let $X$ be a scheme, let $V$ be a vector bundle on $X$ of rank $r$, and let
$\mathcal F = \mathop{\rm Flag}(V)$ be the $X$-scheme of full flags in $V$.
Denote by
\[
F_\bullet^{\rm univ} = (0 = F_0^{\rm univ} \subset F_1^{\rm univ} \subset
\cdots \subset F_r^{\rm univ} = V \times_X \mathcal F)
\]
the universal flag, and for each $i$, by $\mathcal L_i$ the line bundle
$F_i^{\rm univ}/F_{i-1}^{\rm univ}$ on $\mathcal F$. Then the tensor
product
\[
\bigotimes_{i=1}^r \mathcal L_i^{\otimes -\frac{i(i+1)}2}
\]
is very ample (for $\mathcal F \rightarrow X$).
\end{lm}

\begin{proof}
Embed the flag scheme into a product of Grassmannians over $X$, and then
use the Pl\"{u}cker embedding and the Segre embedding to embed it into
projective space over $X$. The pull-back of $\mathcal O(1)$ under this
embedding is the line bundle given in the lemma.
\qed
\end{proof}

The following result was proved by Moret-Bailly; see Theorem 1.1 in the
Introduction of \cite{moret-bailly}.

\begin{lm}\label{qa-lem2}
The line bundle $\bigwedge^{\rm top} \omega(A^{\rm univ})$ on $\mathcal A_g$
is ample.
\end{lm}

\subsection{}
As a special case of Lemma \ref{lm:rel_pos_from_flag} (2), we have

\begin{lm} \label{lm_rhox}
Let $x\in{\rm Adm}(\mu)$, write $x = w t^{\rho(x)}$, and fix a chain
$(A_i)_i$ in the KR stratum associated with $x$.
We write $\rho(x) = (\rho(x)(1), \dots, \rho(x)(2g))\in \mathbb Z^{2g}$.
Furthermore fix  $i \in \{ 1, \dots, 2g \}$.

Then $\alpha(\omega_{i}) \subseteq \alpha(\omega_{i-1}) \subset H^1_{DR}(A_0)$
is a strict inclusion if and only if $\rho(x)(i) = 1$.
\end{lm}

Now we can prove the main result of this section:

\begin{thm} \label{KR_quasiaffine}
All KR strata in $\mathcal A_I$ are quasi-affine.
\end{thm}

\begin{proof}
Let $x \in {\rm Adm}(\mu)$. We construct an ample line bundle on $\mathcal A_x$
which is a torsion element in the Picard group. This shows that the structure
sheaf is ample, which by EGA II Proposition~5.1.2 is equivalent to
$\mathcal A_x$ 
being quasi-affine. For $i \in \{ 1, \dots, 2g \}$, let
\[
\mathcal L_i := H^1_{DR}(A_{i-1})/\alpha(H^1_{DR}(A_i)),
\]
where $(A_\bullet)_\bullet$ denotes the universal chain of abelian schemes,
and where as usual $\alpha$ denotes the
canonical map. So $\mathcal L_i$ is a line bundle on $\mathcal A_x$ (or
even on $\mathcal A_I$). Write $x = w t^{\rho(x)}$, $w\in W$, $\rho(x)\in
X_*(T)$.
\begin{enumerate}
\item
We have
\[
\mathcal L_i \cong \left( \omega(A_{v(i)-1})/\alpha(\omega(A_{v(i)})
\right)^{\varepsilon(i)} 
\]
where
\begin{eqnarray*}
v(i) & = & \left\{ \begin{array}{ll}
w^{-1}(i) & \text{if }\rho(x)(w^{-1}(i)) = 1 \\
2g- w^{-1}(i) +1 & \text{if }\rho(x)(w^{-1}(i)) = 0
\end{array}\right., \\
\varepsilon(i) & = & \left\{ \begin{array}{ll}
1 & \text{if }\rho(x)(w^{-1}(i)) = 1 \\
-1 & \text{if }\rho(x)(w^{-1}(i)) = 0
\end{array}\right. .
\end{eqnarray*}
\item
If $\rho(x)(i) = 1$, then
\[
\mathcal L_i \cong \left( \omega(A_{i-1})/\alpha(\omega(A_i))
\right)^{\otimes p} 
\]
\end{enumerate}
To prove (1), first note that using duality we may restrict to the case
where $\varepsilon(i) = 1$. It is easy to check that the analogue of the
claimed equality holds at the unique $T$-fixed point of the stratum of the
local model corresponding to $x$ (cf.~\cite{goertz-yu}, Subsection 2.6). This
implies that it holds pointwise, at every point of $\mathcal A_x$.
More precisely, we see (from the permissibility condition) that $i \ge
v(i)$, and that the inclusion $\omega(A_{v(i)-1}) \rightarrow
H^1_{DR}(A_{v(i)-1})$ factors through $\alpha(H^1_{DR}(A_i))$. This is true
globally, so that we get a morphism
$\omega(A_{v(i)-1})/\alpha(\omega(A_{v(i)})) \rightarrow \mathcal L_i$
globally on $\mathcal A_x$. It follows that this is an isomorphism by the
pointwise result, and this gives (1).

Furthermore, (2) follows immediately from the fact that the Hodge
filtration is the image of the Verschiebung morphism, more precisely:
\[
\omega(A)^{(p)} = \im(V\colon H^1_{DR}(A) \longrightarrow
H^1_{DR}(A)^{(p)}).
\]
Putting (1), (2) and Lemma~\ref{lm_rhox} together, we obtain first that all the line bundles
$\omega(A_{i-1})/\alpha(\omega(A_i)i)$ (restricted to $\mathcal A_x$) are of
finite order, and using (1) once more, that all the $\mathcal L_i$ are of
finite order. By Lemma~\ref{qa-lem1}, a suitable tensor product of $\mathcal L_i$'s is very ample for the projection from $\mathcal A_x$ to $\mathcal A_g$. The pull-back of the line bundle $\bigwedge^{\rm top} \omega(A^{\rm univ})$ on $\mathcal A_g$ to $\mathcal A_I$ is $\bigwedge^{\rm top} \omega(A_0)$, which can also be expressed as a tensor product of $\mathcal L_i$'s, because $\omega(A_0)$ has a filtration whose subquotients are of the form $\mathcal L_i$. Using Lemma~\ref{qa-lem2}, altogether we obtain an ample invertible sheaf of finite order on $\iota(\mathcal A_x)$, and also on $\mathcal A_x$ since $\iota$ is finite.
\qed
\end{proof}

\subsection{Affineness}
Compare the results about affineness of KR strata: Using Proposition
\ref{rel_to_EvdG}, we get from \cite{ekedahl-vdgeer}, Proposition~10.5
(ii) that 
for $p$ large and $w\in W$ such that $w\tau$ is admissible and has $p$-rank $0$, the KR stratum
$\mathcal A_{w\tau}$ is affine. It would be interesting to generalize this
result to all KR strata of $p$-rank $0$, maybe using the map $\iota$
defined above together with a Pieri formula for the affine flag variety.

In \cite{goertz-yu} we showed that for $p\ge 2g$, every superspecial KR
stratum is affine.

% ========================================================================
\section{The closure of components of KR strata}

\subsection{Sections}

Note that although for any two points in a fixed EO stratum $EO_w$ the
$p$-torsion group schemes are isomorphic, one cannot expect to find a
flat, surjective cover $X \rightarrow EO_w$ such that the pull-back of the
finite flat group scheme $A^{\rm univ}[p]$ to $X$ is constant. However, we
at least have the following lemma (whose first part is taken from
\cite{ekedahl-vdgeer}).

\begin{lm} \label{sections}
Fix an EO type $w\in W_{\rm final}$.
\begin{enumerate}
\item
There exists a flat surjective map $X \rightarrow EO_w$ such that the
data
\[
(H^1_{DR}(\mathbf A^{\rm univ}_{|X}), \omega(\mathbf A^{\rm
univ}_{|X}), F, V, \langle \cdot,\cdot\rangle)
\]
is constant.
\item
Let $X$ be as in (1), let $x \in {\rm Adm}(\mu)$, such that
$\pi(\mathcal A_x) \cap EO_w \ne \emptyset$, and fix $a \in \mathcal
A_x$ such that $\pi(a) \in EO_w$. There is a morphism $f_a\colon X
\rightarrow {\rm Flag}^\perp(\mathbb H)$ of schemes over $\mathcal A_g$,
such that $f_a(X) \subseteq \iota(\mathcal A_x)$ and $\iota(a)\in f_a(X)$.
\end{enumerate}
\end{lm}

\begin{proof}
Part (1) is proved in the proof of Proposition 9.6 (ii) in
\cite{ekedahl-vdgeer}. We prove part (2).
Let $H$ be the pull-back of $H^1_{DR}(A_{\pi(a)})$ to $X$ (where
$A_{\pi(a)}$ denotes the abelian scheme over $k$ corresponding to
$\pi(a)$).
Denote the map $X \rightarrow \mathcal A_g$ by $q$. We have to define, for
every $S$ and every $s\in X(S)$, a flag inside $H^1_{DR}(A_{q(s)})$.
However, by assumption $H^1_{DR}(A_{q(s)})$ is trivial, and we can identify
it with the pull-back of $H$ to $S$. In particular, the
flag in $H$ induced from the point $a$ gives us a flag in
$H^1_{DR}(A_{q(s)})$, and this is the one we use to define the morphism
$f_a$. Clearly, $\iota(a)\in f_a(X)$. Furthermore, the assertion that
$f_a(X) \subseteq \iota(\mathcal A_x)$, which is meant purely
set-theoretically, follows immediately from Lemma \ref{lm:rel_pos_from_flag}.
\qed
\end{proof}

As an application, we have
\begin{lm}\label{62}
Let $x \in {\rm Adm}(\mu)$, let $Z$ be a connected component of
$\mathcal A_x$. Then the image $\pi(Z)$ is a union of connected components
of EO strata. 
\end{lm}
\begin{proof}
Let $w$ be an EO type and $EO_{w}^c$ be a connected component of the EO
stratum $EO_w$, such that $\pi(Z) \cap EO_w^c\ne \emptyset$. We must
show that $EO_w^c\subseteq \pi(Z)$. 
Let $q:X\to EO_{w}^c$ be the restriction of a map 
as in Lemma~\ref{sections} to $EO_{w}^c$.
Let $\{X_i\}$ be the connected components of $X$. 
We show that whenever $\pi(Z)\cap q(X_i)\neq
\emptyset$, one has $q(X_i)\subset \pi(Z)$. 
Let $a\in Z$ such that $\pi(a)\in q(X_i)$, and consider the map 
$f_a \colon X \rightarrow {\rm Flag}^\perp(\mathbb 
H)$ from Lemma \ref{sections}. 
Recall that $\mathbf p$ denotes the projection ${\rm Flag}^\perp(\mathbb
H)\rightarrow \mathcal A_g$.
Since $f_a(X_i)\subset \iota(Z)$ and
$\mathbf p\circ \iota=\pi$, we get $q(X_i)=\mathbf p(f_a(X_i))\subset
\mathbf p(\iota(Z))=\pi(Z)$. 
Since $q$ is flat and surjective, $\{q(X_i)\}$ is an open covering of
$EO_w^c$, and hence $U_1:=\pi(Z)\cap EO_w^c$ is the union of those  
$q(X_i)$ with $q(X_i)\cap \pi(Z)\neq \emptyset$. The scheme $EO_w^c$
is a disjoint union of $U_1$ and the open subset $U_2$
which is the union of those $q(X_i)$ which are disjoint from $\pi(Z)$. The
connectedness of $EO^c_w$ implies that $\pi(Z) \supset EO^c_w$. \qed 
\end{proof}

\begin{remark}\label{63}
  As we know, all non-supersingular EO strata are connected (and we will
  show below that the same is true for all non-supersingular KR strata;
  see Theorem~\ref{ssp_or_irred}). However, in the supersingular case,
  the image 
  $\pi(Z)$ may contain more than one connected components of an EO stratum.
  In the case $g=2$ (see Example~\ref{33}), the image $\pi(Z)$ of a
  connected component $Z$ of $\calA_{s_{02}\tau}$ contains one connected
  component of $EO_{w_{\rm ssi}}$ and $p^2+1$ connected components (which
  simply are points) of $EO_{w_{\rm ssp}}$.  
\end{remark}
\subsection{}

\begin{thm} \label{closure_meets_minimal}
Let $x \in {\rm Adm}(\mu)$, let $Z$ be a connected component of
$\mathcal A_x$, and let $\overline{Z}$ be its closure in $\mathcal A_I$.
Then $\overline{Z} \cap \mathcal A_\tau \ne \emptyset$.
\end{thm}

This was proved in the special case that $x$ is superspecial
(\cite{goertz-yu}, Lemma 6.4). It follows from
\cite{ekedahl-vdgeer}, Theorem~6.1 and the identification of $\mathcal
U_w$ and $\mathcal A_{w\tau}$ that this is true for 
elements $w\in W$ such that $w\tau\in{\rm Adm}(\mu)$. We start with a 
simple lemma which will allow us to prove the theorem by induction.

\begin{lm}
With the notation of the theorem, assume that for some $y \in {\rm
Adm}(\mu)$, the intersection $\mathcal A_y \cap \overline{Z}$ is not empty.
Then this intersection is a union of connected components of $\mathcal
A_y$. 
\end{lm}

\begin{proof}
Assume that $V$ is a connected component of $\mathcal A_y$ which meets
$\overline{Z}$. We have to show that it is contained in
$\overline{Z}$. The assumption implies that the closure $\overline{\mathcal
A}_x$ contains $\mathcal A_y$, so in particular it contains $V$. However,
if $Z$, $Z'$ are different connected components of $\mathcal A_x$, then
their closures do not intersect, because $\overline{\mathcal A}_x$ is
normal.
\qed
\end{proof}

Because of this lemma, the theorem follows from

\begin{prop}
Let $x\in {\rm Adm}(\mu)$ be an element of length $>0$, i.~e.~such that
$\dim \mathcal A_x > 0$. Let $Z$ be a connected component of $\mathcal
A_x$. Then $Z$ is not closed in $\mathcal A_I$.
\end{prop}

\begin{proof}
If $\mathcal A_x$ is contained in the $p$-rank $0$ locus, then $Z$ being
closed would imply that $Z$ is proper, since the $p$-rank $0$ locus is
proper (Proposition~\ref{p_rk_0_proper}). However, since $Z$ is quasi-affine, it
cannot be proper unless it is finite, which contradicts our hypothesis,
because $\dim Z = \dim \mathcal A_x$.

Now assume that $Z$ is closed in $\mathcal
A_I$. We will prove that the $p$-rank on $Z$ is necessarily $0$ in this
case, so that we can conclude by the previous step.

% \emph{Claim.} $\pi(Z)$ is a union of connected components of EO strata.

% To prove the claim, we use Lemma \ref{sections}. Let $w$ be an EO
% type and $EO_{w}^c$ be one of its connected components, such that
% $\pi(Z) \cap EO_w^c\ne \emptyset$. We must show that
% $EO_w^c\subseteq \pi(Z)$. Let $a\in Z$ such that $\pi(a)\in
% EO_w^c$, and consider the map $f_a \colon X \rightarrow {\rm Flag}^\perp(\mathbb
% H)$ from Lemma \ref{sections}. Let $X^c$ be a connected component of $X$
% whose image in ${\rm Flag}^\perp(\mathbb H)$ contains $\iota(a)$. The
% restriction of the map $q\colon X \rightarrow EO_w$ to $X^c$ is still
% flat; in particular $q(X^c)$ is dense in $EO_w^c$. On the other hand,
% $f_a(X^c) \subseteq \iota(Z)$ (since it is connected, lies in $\mathcal
% A_x$ and contains $\iota(a)$). Since $\pi(Z) = p(\iota(Z))$ is closed and
% contains the dense subset $q(X^c)$ of $EO_w^c$, it contains all of
% $EO_w^c$.
% This proves the claim.

By Lemma~\ref{62},  $\pi(Z)$ is a union of connected components of EO strata.
Since $\pi$ is proper and $Z$ is assumed to be closed, $\pi(Z)$ is closed,
too. However, whenever $V$ is a connected component of any EO stratum, its
closure meets the superspecial locus, the unique zero-dimensional EO
stratum. This shows that $\pi(Z)$ meets the superspecial locus, hence the
$p$-rank on $Z$ is $0$.
\qed
\end{proof}

% \subsection{} \label{tau_p}
% \begin{prop} \label{EO_indep_of_comp}
% Let $x \in {\rm Adm}(\mu)$. Then $\pi(\overline{\mathcal A}_x)$ is the
% closure of a (uniquely determined) EO stratum. In other words: all generic
% points of irreducible components of $\mathcal A_x$ map to the same EO
% stratum under $\pi$.
% \end{prop}
% 
% \begin{proof}
% \qed
% \end{proof}
% 
% (Cf.~\cite{ekedahl-vdgeer} Proposition~9.6 for $x \in W\tau$). 
% 
% We obtain a map ${\rm Adm}(\mu) \rightarrow W/S_g$, mapping an element $w$
% to the element $\varphi\in {\bf ES}=W/S_g$ such that the image of the
% closure of the KR stratum for $w$ is the closure of the EO stratum for
% $\varphi$. In loc.~cit.~the restriction of this map to $W\tau$ is denoted
% by $\tau_p$ (indicating that a priori it is not clear whether this map
% depends on the characteristic $p$). It seems to be very interesting (and
% difficult) to describe this map group-theoretically.
% 
% The above proposition implies that the union of all supersingular KR strata
% (in the sense of \cite{goertz-yu}) is contained in the inverse image under
% $\mathcal A_I \rightarrow \mathcal A_g$ of the union of all EO strata which
% are contained in the supersingular locus. We do not know whether this
% inclusion is an equality.

% ========================================================================
\section{Connectedness of KR strata}

The aim of this section is to prove that all non-superspecial KR strata are
connected, or equivalently irreducible. We follow the strategy of proof of
\cite{ekedahl-vdgeer}, Theorem~11.4 (ii).

\subsection{Unions of one-dimensional strata}
Here, and below, we silently exclude the case $g=1$ which is
uninteresting with respect to the results discussed in the sequel, and
would often require a separate treatment. 
We start by considering unions of KR strata of dimension $\le 1$. All of
these are superspecial by~\cite{goertz-yu}, Proposition 4.4, which we
recall here for the convenience of the reader, because it will be used
below in several places: 

\begin{prop}
The KR stratum associated with $x \in {\rm Adm}(\mu)$ is superspecial if
and only if $w \in\bigcup_i W_{\{i,g-i\}}\tau$.

In particular, for all $w \in \bigcup_i W_{\{i,g-i\}}$, the KR stratum
associated with $w\tau$ is supersingular.
\end{prop}

So for all KR strata of dimension $\le 1$  we have the description of
superspecial KR strata given in \cite{goertz-yu} at our disposal. To
make use of this description, we use the following 
lemma (the notation used in its statement is independent from our
notation fixed above).

\begin{lm}
Let $G$ be a connected reductive group over the finite field 
$\mathbb F_q$. Let $T\subset B \subset G$ be a maximal torus 
and a Borel subgroup over $\F_q$. Denote by $W$ the absolute Weyl 
group, and denote by $\sigma$ the 
automorphism of $W$ induced by the Frobenius automorphism $\sigma$ of
$\overline{\mathbb F}_q/\mathbb F_q$ (and $G(\overline{\mathbb F}_q)$
etc.). Let 
$S\subset W$ be the set of simple reflections determined by $B$, and let $I
\subseteq S$ be a subset which is not contained in any proper
$\sigma$-stable subset of $S$. Then the union 
\[
\overline{X(I)} := X({\rm id}) \cup \bigcup_{s\in I} X(s)
\]
of Deligne-Lusztig varieties is connected.
\end{lm}

\begin{proof}
See \cite{goertz:connDL}.
In the case of a unitary group, which is the case we need below, this
result is
proved in \cite{ekedahl-vdgeer}, Lemma 7.6 (ii).
\qed
\end{proof}

As a side note we remark that the lemma implies Lusztig's connectedness
criterion for Deligne-Lusztig varieties (see
e.~g.~\cite{bonnafe-rouquier:irred}) much in the same way as we derive the
connectedness of non-superspecial KR strata below. Namely, we use the
facts that Deligne-Lusztig
varieties are quasi-affine (this implies that no irreducible component
can be closed in $G/B$), and they have normal closures. Then proceed
as in the proof of Theorem \ref{ssp_or_irred}. See
\cite{goertz:connDL} for more details. 

\begin{thm} \label{connectedness_one_diml}
Let $x \in {\rm Adm}(\mu)$, and assume that $x$ is not superspecial. Let
$S(x) \subseteq \{ 0, \dots, g\}$ be the set of indices $i$ such that the
simple reflection $s_i$ is less or equal than $x\tau^{-1}$ with
respect to the 
Bruhat order (in other words: $s_i$ occurs in any, or equivalently: every,
reduced word expression for $x\tau^{-1}$). Then
\[
\mathcal A_{x,1} := \bigcup_{i\in S(x)} \overline{\mathcal A}_{s_i\tau}
\]
is connected, where $\overline{\mathcal A}_{s_i\tau} = \mathcal
A_{s_i\tau}\cup 
\mathcal A_\tau$ is the closure of $\mathcal A_{s_i\tau}$.
\end{thm}

\begin{proof}
Saying that $x$ is not superspecial is equivalent to saying that for every
$i = 0,\dots, [g/2]$, $S(x)$ contains at least one of $s_i$, $s_{g-i}$.
In particular $\mathcal A_{x,1}$ contains $\overline{\mathcal
A}_{s_0\tau}$ or $\overline{\mathcal A}_{s_g\tau}$. Both these have
one-dimensional image in $\mathcal A_g$, because the image, which is a
union of EO strata, strictly contains the $0$-dimensional EO stratum (which
is the set of superspecial points). This implies that in both cases, the
image is the closure of the unique $1$-dimensional EO stratum, which is
connected (as proved by Oort \cite{oort01}, Proposition~7.3). This
implies that 
for any two superspecial points in $\mathcal A_g$, there exists a connected
component of $\overline{\mathcal A}_{s_0\tau}$ (and likewise for $s_g$)
connecting a point of $\mathcal A_\tau$ in the fiber of $\pi\colon \mathcal
A_I\rightarrow \mathcal A_g$ over the first point with a point of $\mathcal
A_\tau$ in the fiber over the second point.

Therefore it is enough to show that any two points of $\mathcal A_\tau$
lying in the same fiber $\pi^{-1}(A_0)$ can be connected by a sequence of
lines in $\mathcal A_{x,1}$. Here $A_0 \in \mathcal A_g(k)$ is some
superspecial abelian variety. We obtain a point $(A_0 \rightarrow A_g)$ of
$\mathcal A_{\{0,g\}}(k)$ by setting $A_g := A_0^{(p)}$, and taking
Frobenius as the isogeny. Denote the projection $\mathcal A_I \rightarrow
\mathcal A_{\{0,g\}}$ by $\pi_{\{0,g\},I}$. We have
\[
\mathcal A_\tau \cap \pi^{-1}(A_0) = \mathcal A_\tau \cap
\pi_{\{0,g\},I}^{-1}((A_0\rightarrow A_g)).
\]
Using the description of $\pi_{\{0,g\},I}^{-1}((A_0\rightarrow A_g))$ given
in \cite{goertz-yu}, Theorem~6.3, we are in the situation of the previous lemma
in the special case of the unitary group over $\mathbb F_p$, given by the
Dynkin diagram of type $A_{g-1}$, on which Frobenius acts by the
non-trivial automorphism (if $g>2$).  The theorem follows.
\qed
\end{proof}

\subsection{Non-superspecial KR strata are connected}

\begin{thm} \label{ssp_or_irred}
Let $x \in {\rm Adm}(\mu)$, and assume that $x$ is not superspecial. Then
$\mathcal A_x$ is irreducible.
\end{thm}

Recall that, in order to avoid technicalities, we equip $\mathcal A_g$
and $\mathcal A_I$ with a full symplectic level $N$-structure with
respect to a fixed primitive $N$-th root of unity, so that these
spaces are connected. 

\begin{proof}
It is equivalent to show that the closure $\overline{\mathcal A}_x$ is
irreducible, and because this closure is normal, it is even enough to show
that it is connected.
By theorem \ref{closure_meets_minimal}, every connected component meets the
minimal KR stratum $\mathcal A_\tau$, hence it meets the locus $\mathcal
A_{x,1}$ defined in Theorem \ref{connectedness_one_diml}. Since $\mathcal
A_{x,1}$ is connected by the theorem and contained in $\overline{\mathcal
A}_x$, $\overline{\mathcal A}_x$ itself is connected.
\qed
\end{proof}

Compare \cite{ekedahl-vdgeer}, Theorem~11.4.
Note that if $w$ has $p$-rank $g$ or $g-1$, this result is known by the work
of the second author \cite{yu:gamma}, \cite{yu:prank}. 
On the other hand, for superspecial KR
strata, there is a formula for the number of connected components in
\cite{goertz-yu}.

\subsection{All supersingular KR strata are superspecial}

As a corollary, we can prove conjecture 4.5 in \cite{goertz-yu}:

\begin{cor} \label{ssi_implies_ssp}
Every KR stratum which is entirely contained in the supersingular
locus, is superspecial, i.~e.~of the form $\mathcal A_x$ for $x \in
\bigcup_{i=0}^{[g/2]} W_{\{i,g-i\}}\tau$.
\end{cor}

\begin{proof}
We may clearly assume that $N$ is large (by passing to a suitable \'etale
extension, if necessary), so that we can assume that all EO strata
contained in the supersingular locus $\mathcal S_g$ are disconnected. 
Let $x\in {\rm Adm}(\mu)$ such that $\mathcal A_x \subseteq \mathcal S_I$.
Its image under the projection $\pi\colon \mathcal A_I \rightarrow \mathcal
A_g$ is a union $\bigcup_{\varphi\in {\bf ES}(x)} EO_\varphi$ of EO strata
by Proposition \ref{KR_maps_to_EO}. Clearly, all these EO strata are entirely
contained in the supersingular locus of $\mathcal A_g$. It follows from
Proposition \ref{properties_EO} that they are not irreducible. Hence the
union cannot be irreducible. This means however that $\mathcal A_x$ is not
irreducible, so by the theorem $\mathcal A_x$ is superspecial.
\qed
\end{proof}

% ========================================================================

\section{The $p$-rank $0$ locus}

\subsection{Group-theoretic notation, II}

We need some more notation. We embed $\Sp_{2g} \subset \SL_{2g}$ in the
standard way (see Section \ref{group_theor_I}). In this setup, the positive
roots are
\[
\begin{array}{lll}
\beta_{ij}^1, & \text{ where } \beta^1_{ij}(k) = \left\{ 
    \begin{array}{ll}
    1, & k = i \text{ or } k = 2g-j+1,\\
    -1, & k = j \text{ or } k = 2g-i+1,\\
    0, & \text{otherwise,}
    \end{array}
    \right. & 1 \le i < j \le g, \\
    \\
\beta_{ij}^2, & \text{ where } \beta^2_{ij}(k) = \left\{ 
    \begin{array}{ll}
    1, & k = i \text{ or } k = j,\\
    -1, & k = 2g-i+1 \text{ or } k = 2g-i+1,\\
    0, & \text{otherwise,}
    \end{array}
    \right. & 1 \le i < j \le g, \\
    \\
\beta^3_i, & \text{ where } \beta^3_{i}(k) = \left\{
    \begin{array}{ll}
    1, & k = i,\\
    -1, & k = 2g-i+1,\\
    0, & \text{otherwise,}
    \end{array}
    \right. & 1 \le i \le g.
\end{array}
\]
So there are $g^2$ positive roots. The simple roots are $\beta_{i,i+1}^1$,
$i=1, \dots, g-1$ and $\beta_g^3$.

We also need the Iwahori-Matsumoto formula for the length of an element in
$\wt W$ which expresses the fact that the length is equal to the number of
affine root hyperplanes separating the alcove in question from the base
alcove. We use the following version:
\begin{equation} \label{length_wtW}
\ell(t^\lambda w) = \sum_{\gfrac{\beta > 0}{w^{-1}\beta > 0}} |
\langle\beta,\lambda\rangle| + \sum_{\gfrac{\beta > 0}{w^{-1}\beta < 0}} |
\langle\beta,\lambda\rangle + 1 |, \quad \lambda\in X_*(T), w \in  W, 
\end{equation}
which is easily checked to be equivalent to formula (2.1) given in
\cite{goertz-yu}.

\subsection{Dimension of the $p$-rank $0$ locus}

Let $W^{(0)} \subset W$ be the subset of elements which have no fixed point,
and let ${\rm Adm}(\mu)^{(0)}$ be the set of admissible elements which give
rise to a stratum on which the $p$-rank is $0$. Proposition \ref{ngo-gen}
shows that the projection $\wt W \rightarrow W$ induces a map
\[
{\rm Adm}(\mu)^{(0)} \rightarrow W^{(0)}.
\]

\begin{lm}
The map ${\rm Adm}(\mu)^{(0)} \rightarrow W^{(0)}$ defined above is a
bijection. Its inverse is given by $w \mapsto t^{\lambda(w)}w$ with
\[
\lambda(w)(i) = \left\{ 
    \begin{array}{ll}
    0, & w^{-1}(i) > i \\
    1, & w^{-1}(i) < i
    \end{array}
    \right., \quad i=1, \dots, 2g.
\]
\end{lm}

\begin{proof}
This is easy to check writing $x$ as an extended alcove (see
\cite{goertz-yu}, Subsections 2.5, 2.6), and is also contained in
\cite{haines:bernstein}, Proof of 8.2. Note however that Haines uses a
different normalization from ours!
\qed
\end{proof}

Now write $W = (\mathbb Z/2)^g \rtimes S_g$, where corresponding to the
embedding $W \subset S_{2g}$, an element $\sigma\in S_g$ corresponds to the
element $w\in W$ with $w(i) = \sigma(i)$, $w(2g-i+1) = 2g-\sigma(i)+1$,
$i=1, \dots, g$, and an element $v \in (\mathbb Z/2)^g$ corresponds to the
permutation $w$ with $w(i) = i$ if $v(i) = 0$, and $w(i) = 2g-i+1$, if
$v(i)=1$, $i=1,\dots, g$.

Fix $\sigma \in S_g$. We say that a vector $v \in (\mathbb Z/2)^g$ is
$\sigma$-admissible, 
if $v\sigma \in W^{(0)}$. Clearly, $v$ is $\sigma$-admissible if and
only if $v(i) = 
1$ for all fixed points $i$ of $\sigma$, so the number of
$\sigma$-admissible vectors 
is $2^{g-f}$, where $f$ is the number of fixed points of $\sigma$.

For $v, v' \in (\mathbb Z/2)^g$, we write $v' \le v$ if $v'(i) \le v(i)$
(where we set $0 < 1$) for all $i$.

\begin{lm}
Let $\sigma \in S_g$, let $v, v' \in (\mathbb Z/2)^g$ be
$\sigma$-admissible, and assume that $v' \le v$.  
Denote by $\lambda(v\sigma)$ the translation element such that
$t^{\lambda(v\sigma)} (v\sigma) \in {\rm Adm}(\mu)^{(0)}$, and
likewise for $v'$.
Then we have $t^{\lambda(v\sigma)} (v\sigma) \le t^{\lambda(v'\sigma)}
(v'\sigma)$ with respect to the Bruhat order.
\end{lm}

\begin{proof}
We may assume that $v$ and $v'$ differ at only one place, say $d \in \{1,
\dots, g\}$, where we have $v'(d) = 0$, $v(d) = 1$. In particular,
$\sigma(d) \ne d$. We write $\lambda = \lambda(v\sigma)$, $\lambda' =
\lambda(v'\sigma)$ and $w = v\sigma$, $w'=v'\sigma\in W$. Our
assumption says precisely that $w' = t_d w$, where $t_d$ is the reflection
$(0, \dots, 0, 1, 0, \dots, 0) {\rm id} \in (\mathbb Z/2)^g \rtimes S_g =
W$ associated with $\beta^3_d$ (the $1$ is in the $d$-th place).

We have $\lambda(d) = 0$, and  $\lambda' = \lambda$ if $\sigma^{-1}(d)>d$ and $\lambda' = t_d
\lambda$ if $\sigma^{-1}(d)<d$.
So we get
\[
t^{\lambda'}w' = \left \{ 
    \begin{array}{ll}
      t^\lambda t_d w, & \sigma^{-1}(d) > d \\
      t^{t_d \lambda} t_d w, & \sigma^{-1}(d) < d
    \end{array}
    \right.
= \left\{ 
    \begin{array}{ll}
      t^\lambda w t_{\sigma^{-1}(d)}, & \sigma^{-1}(d) > d \\
      t_d t^{\lambda}  w, & \sigma^{-1}(d) < d
    \end{array}
    \right.
\]
We see that the two elements $t^{\lambda'}w'$, $t^{\lambda}w$ are related
with respect to the Bruhat order. 

\emph{First case: $\sigma^{-1}(d) < d$.} The two elements differ by the
application of $t_d$ on the left, and we need to check which of the
elements is on the same side of the wall corresponding to $\beta^3_d$ as the
base alcove (this will be the smaller element). Since $\langle \beta^3_d,
\lambda \rangle = -1$, we see that the alcove corresponding to $t^\lambda
w$ is the smaller of the two (recall our normalization of putting the base
alcove in the \emph{anti-dominant} chamber).

\emph{Second case: $\sigma^{-1}(d) > d$.} Instead of comparing the two
elements directly, we compare their inverses. One easily computes that
$(t^{\lambda'}w')^{-1}=(t^\lambda w
t_{\sigma^{-1}(d)})^{-1}=t_{\sigma^{-1}(d)} t^{-w^{-1}\lambda}
w^{-1}$. It suffices to check that $\langle \beta^3_{\sigma^{-1}(d)},
- w^{-1} \lambda \rangle<0$.  
% which differ by $t_{\sigma^{-1}(d)}$ applied on the left. 
% The translation part of
% $t^\lambda w$ is $- w^{-1} \lambda$, and 
We have
\begin{equation*}
  \begin{split}
   \langle \beta^3_{\sigma^{-1}(d)}, - w^{-1} \lambda \rangle & = 
-\lambda(w(\sigma^{-1}(d)))+\lambda(2g-w(\sigma^{-1}(d))+1) \\ & =
-\lambda(2g-d+1)+\lambda(d) = -1 < 0 
  \end{split}
\end{equation*}
%\[
% \langle \beta^3_{\sigma^{-1}(d)}, - w^{-1} \lambda \rangle = 
% -\lambda(w(\sigma^{-1}(d)))+\lambda(2g-w(\sigma^{-1}(d))+1) =
% -\lambda(2g-d+1)+\lambda(d) = -1 < 0
%\]
as desired.
\qed
\end{proof}

As a consequence of the two lemmas, we have the following description.

\begin{prop}
Let $x \in {\rm Adm}(\mu)^{(0)}$ be an admissible element of $p$-rank $0$
which is maximal with respect to the Bruhat order within this set. Let
$\sigma$ be the $S_g$-component of its image in $W^{(0)}$. Then $x
= t^{\lambda_\sigma} (v_\sigma \sigma)$ with 
\[
v_\sigma(i) = \left\{ 
    \begin{array}{ll}
    1, & \sigma(i) = i \\
    0, & \sigma(i) \ne i
    \end{array}
    \right., \quad i=1, \dots, g.
\]
and
\[
\lambda_\sigma(i) = \left\{ 
    \begin{array}{ll}
    0, &  \sigma^{-1}(i) \ge i\\
    1, &  \sigma^{-1}(i) < i
    \end{array}
    \right., \quad i=1, \dots, 2g.
\]
\end{prop}

To express the length of elements of the form $t^{\lambda_\sigma} (v_\sigma
\sigma)$, we make the following definition:

\begin{defn}
Let $\sigma \in S_g$. We define
\[
A_\sigma = \# \{ (i,j) \in \{ 1, \dots, g\}^2;\ i < j < \sigma(j) <
\sigma(i) \}.
\]
\end{defn}

We have the following estimate involving the length of $\sigma$ and the
quantities $A_\sigma$, $A_{\sigma^{-1}}$.  This is an elementary statement
about the symmetric group. The proof we give is quite intricate, even if
entirely elementary. It would be interesting to formulate it in a way which
generalizes to Weyl groups of arbitrary reductive groups; maybe that would
also lead to a more conceptual proof.

\begin{lm}\label{estimate_A_sigma}
Let $\sigma \in S_g$. Then
\[
\ell(\sigma) - 2(A_\sigma+A_{\sigma^{-1}}) \ge  \frac{g- \#\{i;\
\sigma(i)=i\}}{2}.
\]
\end{lm}

\begin{proof}
It is not hard to derive this statement from the results of Clarke,
Stein\-gr\'im\-sson and Zeng \cite{clarke-steingrimsson-zeng}, in particular
the statement ${\rm INV}_{\rm MV} = {\rm INV}$; see
\cite{clarke-steingrimsson-zeng}, Proposition 9. To save the reader
the work of tracing through the notation of 
\cite{clarke-steingrimsson-zeng}, 
we will explain how to reduce our claim to the following lemma,
which is Lemma 8 in \cite{clarke-steingrimsson-zeng}; see also Lemma 3
in Clarke \cite{clarke}.

\begin{lm}
Let $\sigma \in S_g$ be a permutation. Write $a_i := \sigma(i)$. Then
\begin{eqnarray*}
&& \# \{ (i,j);\ i\le j<a_i,\ a_j>j \} = \# \{ (i,j);\ a_i<a_j\le i,\
a_j> j \}, \\
&& \# \{ (i,j);\ i\le j<a_i,\ a_j\le j \} = \# \{ (i,j);\ a_i<a_j\le i,\
a_j\le j \}.
\end{eqnarray*}
\end{lm}

Using the lemma, let us prove Lemma \ref{estimate_A_sigma}. Let
$\sigma\in S_g$, and 
again write $a_i = \sigma(i)$ to shorten the notation. Writing
$\ell(\sigma)$ as the number of inversions, i.~e.~as the number of pairs
$(i,j)$ such that $i<j$, $a_i>a_j$, we have
\begin{eqnarray*}
\ell(\sigma) - A_\sigma - A_{\sigma^{-1}} & = & \# \{ (i,j);\ i<j,\ a_i >
a_j,\ i\le a_i,\ j\ge a_j \} \\[.15cm]
& = & \# \{ (i,j);\ i<j,\ a_i>a_j,\ i\le a_i,\ j\ge a_j,\ j\ge a_i \} \\
& & + \# \{ (i,j);\ i<j,\ a_i>a_j,\ i\le a_i,\ j\ge a_j,\ j<a_i \}
     \\[.15cm]
& = &   \# \{ (i,j);\ j \ge a_i > a_j, i \le a_i \}  \\
& & + \# \{ (i,j);\ i<j<a_i,\ j \ge a_j \}
\\[.15cm]
& \ge & \# \{ (i,j);\ j \ge a_i > a_j,\ i < a_i \}  \\
& & +\# \{ (i,j);\ i<j<a_i,\ j \ge a_j \}
\\[.15cm]
& = & \# \{ (i,j);\ i\le j<a_i,\ a_j>j \} \\
& & + \# \{ (i,j);\ a_i<a_j\le i,\ a_j\le j \}
\\[.15cm]
&\ge &  \# \{ i;\ i < a_i \} + \# \{ (i,j);\  i < j < a_j < a_i \}\\
& & + \# \{ (i,j);\ a_i<a_j<j<i \}
\\[.15cm]
& = &  \# \{ i;\ i < a_i \} + A_\sigma + A_{\sigma^{-1}}.
\end{eqnarray*}
So, for every $\sigma\in S_g$, we get
\[
\ell(\sigma) - 2(A_\sigma + A_{\sigma^{-1}}) \ge \# \{ i;\ i < \sigma(i)
\}.
\]
Since the quantity on the left is the same for $\sigma$ and $\sigma^{-1}$,
we even get that it is greater or equal than
\[
\max (\# \{ i;\ i < \sigma(i) \}, \# \{ i;\ i < \sigma^{-1}(i) \}).
\]
But
\[
\# \{ i;\ i < \sigma(i) \}+ \# \{ i;\ i < \sigma^{-1}(i) \} = g - \#\{i;\
\sigma(i)=i \},
\]
and we finally get the desired inequality.
\qed
\end{proof}

The next lemma, whose proof is unfortunately quite technical and long, is
the heart of the computation of the dimension of $\mathcal A_I^{(0)}$.

\begin{lm} \label{dim_lemma}
For $\sigma\in S_g$, let $v_\sigma$, $\lambda_\sigma$ be as in the
proposition.
\begin{enumerate}
\item % 1
For all $\sigma$, we have
\[
\ell(t^{\lambda_\sigma} (v_\sigma\sigma)) = g(g+1)/2 + 2 A_\sigma +
2A_{\sigma^{-1}}  - \ell(\sigma) - \# \{ i;\ \sigma(i) = i \}.
\]
\item % 2
Furthermore, for all $\sigma$,
\[
\ell(t^{\lambda_\sigma} (v_\sigma\sigma)) \le [ g^2/2 ].
\]
\item % 3
Let $\sigma = (12)(34) \cdots (g-1,g)$ if $g$ is even, and
$\sigma = (12)(34) \cdots (g-2,g-1)$ if $g$ is odd. Then
\[
\ell(t^{\lambda_\sigma} (v_\sigma\sigma)) = [ g^2/2 ].
\]
\end{enumerate}
\end{lm}

\begin{proof}
We prove \textbf{part (1)} using the Iwahori-Matsumoto formula
(\ref{length_wtW}).
The proof is not hard, but quite long, so we restrict ourselves to
illustrating the method by discussing a few cases. We write $w =
v_\sigma \sigma$,
$\lambda = \lambda_\sigma$ and $x = t^\lambda w$.

\emph{Contributions to the sum from roots $\beta^1_{ij}$.} 
First, we need to check when $w^{-1} \beta^1_{ij}$ is positive. We get
\[
w^{-1} \beta^1_{ij} \left\{ 
    \begin{array}{ll}
     > 0, & v_\sigma(i) = 0, \text{and } (v_\sigma(j)=1 \text{ or }
     \sigma^{-1}(i) < \sigma^{-1}(j)),  \\ 
     < 0, & \text{otherwise}.
    \end{array}
    \right.
\]
Now, for each of these cases, we compute the contributions to the length of
$x$. Assume that $v_\sigma(i) = v_\sigma(j) = 0$. So $w^{-1} \beta^1_{ij} >
0$ if and only if $\sigma^{-1}(i) < \sigma^{-1}(j)$. We have
\[
\langle \beta^1_{ij}, \lambda \rangle = \lambda(i) - \lambda(j) = 
\left\{
    \begin{array}{ll}
    1, & \sigma^{-1}(i) < i, \sigma^{-1}(j) > j, \\
    0, &  \sigma^{-1}(i) > i, \sigma^{-1}(j) > j \text{ or }
          \sigma^{-1}(i) < i, \sigma^{-1}(j) < j,  \\
    -1, & \sigma^{-1}(i) > i, \sigma^{-1}(j) < j.
    \end{array}
    \right.
\]
The overall contribution to the final sum is (we always sum over $i, j \in \{1,
\dots, g\}$)
\begin{eqnarray*}
&& \# \{ (i,j);\ i<j,\ \sigma^{-1}(i) < i, \sigma^{-1}(j) > j, \sigma^{-1}(i)
< \sigma^{-1}(j) \} \\
& + & \# \{ (i,j);\ i<j,\ \sigma^{-1}(i) > i, \sigma^{-1}(j) < j,
\sigma^{-1}(i) < \sigma^{-1}(j) \} \\
& + & \# \{ (i,j);\ i<j,\ \sigma^{-1}(i) > i, \sigma^{-1}(j) > j,
\sigma^{-1}(i) > \sigma^{-1}(j) \} \\
& + & \# \{ (i,j);\ i<j,\ \sigma^{-1}(i) < i, \sigma^{-1}(j) < j,
\sigma^{-1}(i) > \sigma^{-1}(j). \}
\end{eqnarray*}
One gets similar (but a little simpler) contributions from the other cases.

\emph{Contributions to the sum from roots $\beta^2_{ij}$.} This is similar
to the case of $\beta^1_{ij}$.

\emph{Sum of contributions from roots $\beta^1_{ij}$, $\beta^2_{ij}$.}
Summing up all the contributions we have so far, one gets
\begin{eqnarray*}
&& \# \{ (i,j);\ \sigma^{-1}(i) < \sigma^{-1}(j) \} \\
& + & 2 \cdot \# \{ (i,j);\ \sigma^{-1}(j) < \sigma^{-1}(i) < i < j \} \\
& + & 2 \cdot \# \{ (i,j);\ i < j < \sigma^{-1}(j) < \sigma^{-1}(i) \}.
\end{eqnarray*}
Since 
\[
\ell(\sigma) = \ell(\sigma^{-1}) = \# \{ (i,j);\ i < j,\
\sigma^{-1}(i)>\sigma^{-1}(j) \},
\]
the term in the first row is just $\frac{g(g-1)}{2} - \ell(\sigma)$. In the
second row we have $2A_\sigma$, and in the third row $2A_{\sigma^{-1}}$, so
summing up we get, as the contribution from roots $\beta^1_{ij}$,
$\beta^2_{ij}$:
\begin{equation} \label{eq1}
\frac{g(g-1)}{2} - \ell(\sigma) + 2A_\sigma + 2A_{\sigma^{-1}}.
\end{equation}

\emph{Contributions to the sum from roots $\beta^3_{i}$.} 
We have $w^{-1}\beta^3_i > 0$ if and only if $v_\sigma(i) = 0$, and only
this case gives a contribution. Since $|\langle \beta^3_i, \lambda \rangle
| = 1$, independently of $i$, the contribution we get is
\begin{equation} \label{eq2}
\# \{ i;\ \sigma(i) \ne i \} = g - \# \{ i;\ \sigma(i) = i \}.
\end{equation}

\emph{Summing up.}
Summing up the terms in (\ref{eq1}) and (\ref{eq2}), we get the desired
result.

The estimate in \textbf{part (2)} of the lemma follows from the formula
established in (1) and Lemma \ref{estimate_A_sigma}.

To prove \textbf{part (3)}, we apply the formula in part (1). We have
$\ell(\sigma) = [g/2]$, $A_\sigma = A_{\sigma^{-1}} = 0$, and the number of
fixed points is $0$ if $g$ is even, $1$ if $g$ is odd. So
\[
\ell(t^{\lambda_\sigma} (v_\sigma\sigma)) = \frac{g(g+1)}{2} -
\left[\frac{g+1}{2}\right] = \left[\frac{g^2}{2}\right].
\]
\qed
\end{proof}

Altogether we get a formula for the dimension of $\mathcal A_I^{(0)}$.

\begin{thm} \label{dim_prk0}
The dimension of the $p$-rank $0$ locus is
\[
\dim \mathcal A_I^{(0)} = \max_{x \in {\rm Adm}(\mu)^{(0)}} \ell(x) =
\left[\frac{g^2}{2}\right].
\]
\end{thm}

\subsection{Comparison of superspecial KR locus and $p$-rank $0$ locus}

We recall the following result from \cite{goertz-yu} (Proposition~4.6).

\begin{prop}
The dimension of the union of all superspecial KR strata is 
$g^2/2$ if $g$ is even, and $g(g-1)/2$ if $g$ is odd. There is a unique 
superspecial KR stratum of this maximal dimension, namely the one
corresponding to $w\tau$, where $w$ is the longest element of $W_{\{g/2\}}$, if
$g$ is even, and $w$ is the longest element of $W_{\{0,g\}}$ if $g$ is odd. 
\end{prop}

Combining this with the theorem of the previous section, we have

\begin{cor} \label{dim_ss_locus}
\hspace*{1cm}

\begin{enumerate}
\item
If $g$ is even, the dimension of the
supersingular locus inside $\mathcal A_I$ is
\[
\dim \mathcal S_I = g^2/2.
\]
\item
If $g$ is odd, then
\[
(g^2-g)/2 \le \dim \mathcal S_I \le (g^2-1)/2.
\]
\end{enumerate}
\end{cor}

\subsection{Top-dimensional components of $\mathcal S_I$}

Now let us look at the top-dimensional irreducible components of the union
of superspecial KR strata, and in particular of the supersingular locus
$\mathcal S_I$.

\begin{prop} \label{irred_comp_ss}
Let $X$ be an irreducible component of the closure of the
maximal-dimensional superspecial KR stratum. Then $X$ is an irreducible
component of $\mathcal A_I^{(0)}$, and hence in particular an irreducible
component of $\mathcal S_I$.
\end{prop}

\begin{proof}
Let us first assume that $g$ is odd. 
The maximal-dimensional superspecial KR stratum is $\mathcal A_{w\tau}$,
where $w$ is the longest element of $W_{\{0,g\}} = S_g$, so
\[
w \tau = t^{(0^{(g)}, 1^{(g)})} w_0,
\]
where $w_0 \in W$ is the longest element of $W$. We must show that
$x:=w\tau$ 
is maximal in ${\rm Adm}(\mu)^{(0)}$ with respect to the Bruhat order.
So assume that $x < x'$ for some $x'\in \wt W$, with $\ell(x') =
\ell(x) + 1$. By 
(a suitable) definition of the Bruhat order, this means that $x' = tx$, for
some reflection $t \in W_a$. Because $x\tau^{-1}$ is the longest element of
$W_{\{0,g\}}$, $t$ must be conjugate to $s_0$ or to $s_g$. This implies
that the finite part $t_{\rm fin}$ of $t$ acts as some transposition $(i,
2g-i+1)$. But then the finite part of $x'$ will have a fixed point, so
even if it is admissible, it cannot have $p$-rank $0$.

If $g$ is even, the result follows from a similar consideration, or, even
easier, directly from the dimension counts.
\qed
\end{proof}

In case $g$ is even, we can prove that every top-dimensional irreducible
component of $\mathcal S_I$ is an irreducible component of the closure of
the maximal-dimensional superspecial KR stratum.

\begin{prop} \label{irred_comp_ss_g_even}
Let $g$ be even.  Then every top-dimensional irreducible component of the
supersingular locus is an irreducible component of the union of
superspecial KR strata. More precisely, it is an irreducible component of
the closure of the unique maximal-dimensional superspecial KR stratum
$\mathcal A_{w\tau}$, where $w$ is the longest element of $W_{\{g/2\}}$.
\end{prop}

\begin{proof}
Let $Z$ be an irreducible component of the supersingular locus of maximal
dimension. By Corollary~\ref{dim_ss_locus}~(1), $Z$
has dimension $g^2/2$, and is hence an irreducible component of the
$p$-rank $0$ locus. In particular, it is an irreducible component of the
closure of some KR stratum $\mathcal A_x$. By 
Theorem~\ref{ssp_or_irred}, if $\mathcal A_x$ is not superspecial,
then it is irreducible---a contradiction. 
\qed
\end{proof}

We do not expect the analogous statement to be true for odd $g$. Evidence
from the theory of affine Deligne-Lusztig varieties predicts that it fails
even for $g=3$.

% ========================================================================

\section{The relationship to the work of Ekedahl and van der Geer}
\label{sec:rel_to_EvdG}

\subsection{}
As in \cite{ekedahl-vdgeer}, and as above, let $\mathbb H = H^1_{DR}(A^{\rm
univ}/\mathcal A_g)$, where $A^{\rm univ} \rightarrow \mathcal A_g$ is the
universal abelian scheme.  Let $\mathbb E \subset \mathbb H$ be the Hodge
filtration; this is locally a direct summand of rank $g$, which is totally
isotropic.  We can extend any flag in $\mathbb E$ in a unique way to a
symplectic flag in $\mathbb H$. In this way, we embed the bundle $\mathcal
F:={\rm Flag}(\mathbb E)$ of flags in $\mathbb E$ into ${\rm
Flag}^\perp(\mathbb H)$. The fibers of ${\rm Flag}(\mathbb E)$ over
$\mathcal A_g$ are flag varieties for $\SL_g$. The space ${\rm Flag}(\mathbb
E)$ is denoted $\mathcal F_g$ in \cite{ekedahl-vdgeer}.

Let $(E_\bullet)_\bullet$ be a point of $\mathcal F$. Ekedahl and van der
Geer (\cite{ekedahl-vdgeer}, 3.1) define its conjugate flag as the
unique point 
$(D_\bullet)_\bullet\in {\rm Flag}^\perp(\mathbb H)$ with
\[
D_{g+i} = V^{-1}(E_i^{(p)}).
\]

We get a stratification
\[
\mathcal F = \coprod_{w\in W} \mathcal U_w
\]
by locally closed subsets, given by the relative position of the flag to
its ``conjugate flag'' (both considered as points of ${\rm
Flag}^\perp(\mathbb H)$). Denoting a flag and its conjugate flag by
$(E_\bullet)_\bullet$, $(D_\bullet)_\bullet$ as above, we use
$(D_\bullet)_\bullet$ as the base point to determine the relative position,
i.~e.~with our usual notation, $w = \mathop{\rm inv}((D_\bullet)_\bullet,
(E_\bullet)_\bullet)$. Here $W$, as above, denotes the finite Weyl group
of the symplectic group $\Sp_{2g}$. This stratification is similar to the
stratification of the flag variety by Deligne-Lusztig varieties: however here the
Frobenius varies along the base, and we also have the shift by $g$
(i.~e.~$D_{g+i}$ is defined in terms of $E_i$).

\subsection{The map $\overline{\mathcal A}_{t^\mu} \rightarrow \mathcal F$}

The closure $\overline{\mathcal A}_{t^\mu}$ of the KR stratum corresponding
to $t^\mu$ is an irreducible component of $\mathcal A_I$ (see
\cite{yu:gamma}; see also Theorem \ref{ssp_or_irred} above). We construct a
closed embedding $\overline{\mathcal A}_{t^\mu} \rightarrow \mathcal F$.
First note the following lemma which characterizes the KR strata inside
$\overline{\mathcal A}_{t^\mu}$.

\begin{lm}
\[
Wt^\mu \cap {\rm Adm}(\mu) = W\tau \cap {\rm Adm}(\mu) = \{ x \in \wt W;\ x
\le t^\mu \}.
\]
\end{lm}

\begin{proof}
The decomposition of $\tau$ according to the decomposition $\wt W = X_*(T)
\rtimes W$ is $\tau = w_\emptyset t^{\mu}$, so $t^\mu = w_\emptyset \tau$.
This proves the first equality.  Further, it shows that for $x \in
W_a\tau$, we have $x \le t^\mu$ if and only if $x\tau^{-1} \le
w_\emptyset$, by the definition of the Bruhat order on $\wt W$. It is now
clear that the right hand side is contained in the left hand side.  To see
the converse, use that for every $x \in {\rm Adm}(\mu)$, $x \le
t^{\rho(x)}$, where $\rho(x)\in X_*(T)$ is such that $x = w t^{\rho(x)}$,
$w\in W$. This is proved in \cite{haines:drinfeld}, Proof of Proposition~4.6
(it amounts to the validity of the statement called ${\rm Hyp}(\lambda)$).
Note that at this point it is irrelevant that Haines uses a slightly
different normalization (he puts the base alcove in the dominant chamber).
\qed
\end{proof}

{}From this combinatorial statement, we can derive the following
characterization in terms of abelian varieties.

\begin{prop} \label{char_A_t_mu}
Let $K$ be an algebraically closed field, and let $\mathbf A=(A_i)_i$ be a $K$-valued
point of $\mathcal A_I$. Then the following are equivalent:
\begin{enumerate}
\item[(i)]
We have $\mathbf A \in \ol{\mathcal A}_{t^\mu}$.
\item[(ii)]
We have $\im(H^1_{DR}(A_{2g})\rightarrow
H^1_{DR}(A_g)) = \omega_g$, the Hodge filtration inside $H^1_{DR}(A_g)$.
\item[(iii)]
There is an isomorphism $A_g \cong A_0^{(p)}$ which identifies the given
isogeny $A_0 \rightarrow A_g$ with the Frobenius morphism $A_0\rightarrow
A_0^{(p)}$.
\end{enumerate}
If these conditions are satisfied, then 
\begin{enumerate}
\item
The images $\alpha(\omega_i)
\subset H^1_{DR}(A_0)$ of the Hodge filtrations of $A_i$, $i=0, \dots, g$,
form a complete flag inside $\omega_0$, i.~e.~all the inclusions
$\alpha(\omega_{i+1})\subseteq \alpha(\omega_i)$ are strict.
\item
The finite group scheme $\ker A_0 \rightarrow A_g$ is connected.
\end{enumerate}
\end{prop}

\begin{proof}
To see that conditions (i) and (ii) are equivalent, it is easiest to use
the extended alcove notation (see \cite{goertz-yu}). Say $\mathbf A$ lies
in the KR stratum $\mathcal A_x$.  Condition (i) is equivalent to $x \le
t^\mu$.  By the previous lemma, this holds if and only if $x\in W\tau$, and
this is easily seen to be equivalent to $x_g = (0,\dots, 0)$ (where $x =
(x_0,\dots, x_{2g})$ as an extended alcove). So in terms of lattice chains,
this says that the $g$-th element of the chain corresponding to $\mathbf A$
is the standard lattice, and this means precisely that condition (ii) is
satisfied.

We have a natural identification of $A_0$ and $A_{2g}$, and hence of the
Verschiebung morphism (which induces the Verschiebung of $A_0$ in
cohomology) with a morphism $A_0^{(p)} \rightarrow A_{2g}$. The Hodge
filtration is the image of Verschiebung, i.~e.~$\omega(A_0)^{(p)} = \im
(H^1_{DR}(A_{2g})\rightarrow H^1_{DR}(A_0^{(p)}))$.  Hence condition (iii)
implies condition (ii). Now assume that the latter condition is satisfied.
Consider the \dieu modules of $A_g$ and $A_0^{(p)}$ as submodules of the
\dieu module $M(A_0)$ of $A_0$. For both of them we know that the image of
Verschiebung is equal to $pM(A_0)$. This implies that they coincide, and
therefore condition (iii) holds true.

If these equivalent conditions are satisfied, then (1) follows immediately
from \ref{lm:rel_pos_from_flag} (2). Alternatively, one can use the
numerical characterization of KR strata in terms of abelian varieties; see
\cite{goertz-yu}, Corollary~2.7. Finally, (2) is an obvious
consequence of (iii). 
\qed
\end{proof}

This proposition shows that $\overline{\mathcal A}_{t^\mu}$ is the
subscheme denoted by $\mathcal S(g,p)^\circ$ in \cite{ekedahl-vdgeer}.
See also Lemma \ref{lm_rhox} for a more precise version of (1).

\begin{defn}
We define a morphism $\mathbf i \colon \ol{\mathcal A}_{t^\mu} \rightarrow
\mathcal F$ as follows.  To an $S$-valued point $((A_i)_i, \lambda_0,
\lambda_g)$, we associate the element of $\mathcal F(S)$ given by $(A_0,
\lambda_0)$ and the following flag inside $\omega_0:=\omega(A_0)$:
\[
0 = \alpha(\omega(A_g)) \subset \alpha(\omega(A_{g-1})) \subset \cdots \subset
\alpha(\omega(A_1)) \subset \omega_0.
\]
\end{defn}

The proposition shows that this indeed defines a morphism. Note that it is
different from the morphism $\iota$ defined above. The following lemma
shows that this is the same map as the map $\mathcal S(g,p)^\circ
\rightarrow \mathcal F$ considered in \cite{ekedahl-vdgeer}, Section 14.

\begin{lm}\label{lm:describe_i}
We can also describe the map $\mathbf i$ as follows: a point $\mathbf A$ as above is
mapped to the unique symplectic flag extending the inverse image of
\[
0 = \Lie H_0 \subset \Lie H_1 \subset \cdots \subset \Lie H_g = \Lie A_0
\]
in $H^1_{DR}(A_0)$ under the projection $H^1_{DR}(A_0) \cong
H^1_{DR}(A_0^\vee) \rightarrow \mathop{\rm Lie} A_0$. Here $H_i$ is defined
as the kernel of the isogeny $A_0 \rightarrow A_i$, and we use the given
principal polarization of $A_0$ to identify $H^1_{DR}(A_0)$ and
$H^1_{DR}(A_0^\vee)$.
\end{lm}

\begin{proof}
This follows from the equalities
\begin{eqnarray*}
\Lie H_i & = & \ker(\Lie A_0 \rightarrow \Lie A_i) \\
& = & \ker(H^1_{DR}(A_0^\vee) \rightarrow \Lie A_i)/\omega(A_0^\vee) \\
& = & \alpha^{-1}(\omega(A_i^\vee))/\omega(A_0^\vee) \qquad \text{with }
\alpha\colon H^1_{DR}(A_0^\vee) \rightarrow H^1_{DR}(A_i^\vee)\\
& = & \alpha(\omega(A_i))^\perp/\omega(A_0) \qquad \text{with }
\alpha\colon H^1_{DR}(A_i) \rightarrow H^1_{DR}(A_0)
\end{eqnarray*}
\qed
\end{proof}

Let $\mathcal F' \subset \mathcal F$ denote the closed subscheme of
$V$-stable flags. 

\begin{lm}
We have $\overline{\mathcal U}_{w_\emptyset}\subseteq (\mathcal F')_{\rm
red}$, where $\cdot_{\rm red}$ denotes the underlying reduced subscheme,
and as before $w_\emptyset\in W$ denotes the finite part of $\tau$.
\end{lm}

\begin{proof}
This (and even $\overline{\mathcal U}_{w_\emptyset}=(\mathcal F')_{\rm
  red}$) follows 
from \cite{ekedahl-vdgeer}, Proposition~4.3 (iii).
\qed
\end{proof}

The proof of the following proposition shows in particular that $\mathcal
F'$ is reduced, and is hence equal to $\ol{\mathcal U}_{w_\emptyset}$.

\begin{prop} \label{rel_to_EvdG}
The map $\mathbf i$ is a closed embedding which identifies $\ol{\mathcal
A}_{t^\mu}$ with $\ol{\mathcal U}_{w_\emptyset}$.

For each $w\in W$, we have $w\le w_\emptyset$ if and only if $w\tau \in
{\rm Adm}(\mu)$, and in this case the isomorphism $\ol{\mathcal
A}_{t^\mu}\cong \ol{\mathcal U}_{w_\emptyset}$ restricts to an isomorphism
\[
\mathcal A_{w\tau} \cong \mathcal U_w.
\]
\end{prop}

\begin{proof}
We will show that $\mathbf i$ induces an isomorphism $\ol{\mathcal
A}_{t^\mu} \isomarrow \mathcal F'$.  First, it is clear that $\mathbf i$
factors through $\mathcal F'$.

There is an inverse morphism $\mathcal F' \rightarrow \ol{\mathcal
A}_{t^\mu}$: For a \emph{connected} finite flat group scheme $G$ of
height $1$ (i.~e.~such that the Frobenius morphism is zero), every
$V$-stable subspace of $\mathop{\rm Lie} G$ gives rise to a subgroup
scheme, and hence we can use the description of the morphism $\mathbf i$
given in Lemma~\ref{lm:describe_i} to obtain the desired inverse. In
particular, we see 
that $\mathcal F'$ is reduced.

To prove the compatibility between the stratifications note that
restricting the morphism $\mathcal F \rightarrow {\rm Flag}^\perp(\mathbb
H)$ which maps a flag in $\mathbb E$ to its conjugate flag to $\mathbf
i(\overline{\mathcal A}_{t^\mu})$ gives us precisely the morphism
\[
\overline{\mathcal A}_{t^\mu} \rightarrow {\rm Flag}^\perp(\mathbb H),\quad
(A_\bullet)_\bullet \mapsto (H^1_{DR}(A_{2g-\bullet}))_\bullet.
\]
Now we can use Lemma \ref{lm:rel_pos_from_flag} (3).
\qed
\end{proof}

\subsection{Relation between KR stratification and EO stratification}

Combining Proposition \ref{rel_to_EvdG} and \cite{ekedahl-vdgeer}, 
Corollary~8.4 (iii), we have

\begin{cor}
Let $w \in W_{\rm final}$. Then the morphism $\pi$ restricts to a finite
\'etale surjective morphism $\mathcal A_{w\tau} \rightarrow EO_w$.
\end{cor}

% ========================================================

\begin{thank}
We thank Gerard van der Geer for emphasizing that there is a relationship
between our previous paper \cite{goertz-yu} and the article
\cite{ekedahl-vdgeer} he had written with Ekedahl, and from which we learnt
a lot.
We also thank Michael Rapoport for his helpful remarks on a preliminary
version of this paper.
Most of the results were obtained, and much of the paper was written during
the stay of the first author at Academia Sinica in March 2008. He would
like to thank Academia Sinica for its hospitality, generous support, and
for providing an excellent working environment. The stay was also supported
by the SFB/TR 45 \emph{Periods, moduli spaces, and arithmetic of
algebraic varieties}.
\end{thank}

% --------------------------------------------------------------------------------

\end{document}